\documentclass[a4paper, 11pt]{article}
\usepackage{amsmath}
\usepackage{amssymb,esint}
\usepackage{amscd}
\usepackage{xspace}
\usepackage{color}
\usepackage{authblk}
\usepackage{srcltx} 

\newtheorem{theorem}{Theorem}[section]

\newtheorem{lemma}[theorem]{Lemma}

\newtheorem{proposition}[theorem]{Proposition}

\newenvironment{proof}[1][Proof]{\textbf{#1.} }{\hfill\rule{0.5em}{0.5em}}
{\catcode`\@=11\global\let\AddToReset=\@addtoreset
\AddToReset{equation}{section}

\AddToReset{theorem}{section}

\def\nc{\newcommand}

\def\div{\text{div}}

\nc\pa{\partial}

\nc\CC{\mathbb{C}}
\nc\RR{\mathbb{R}}
\nc\QQ{\mathbb{Q}}
\nc\ZZ{\mathbb{Z}}
\nc\NN{\mathbb{N}}

\title{New gradient estimates for solutions to quasilinear divergence form elliptic equations with general Dirichlet boundary data}
\author{M.-P. Tran\footnote{Corresponding author.}\thanks{Applied Analysis Research Group, Faculty of Mathematics and Statistics, Ton Duc Thang University, Ho Chi Minh city, Vietnam; \texttt{tranminhphuong@tdtu.edu.vn}} , T.-N. Nguyen\thanks{Department of Mathematics, Ho Chi Minh City University of Education, Ho Chi Minh city, Vietnam}}

\date{\today}

\begin{document}
 
\maketitle
\begin{abstract}
This paper studies a new gradient regularity in Lorentz spaces for solutions to a class of quasilinear divergence form elliptic equations with nonhomogeneous Dirichlet boundary conditions:
\begin{align*}
\begin{cases}
\div(A(x,\nabla u)) &= \ \div(|F|^{p-2}F) \quad \text{in} \ \ \Omega, \\
\hspace{1.2cm} u &=\ \sigma \qquad \qquad \qquad \text{on} \ \ \partial \Omega.
\end{cases}
\end{align*}
where $\Omega \subset \mathbb{R}^n$ ($n \ge 2$), the nonlinearity $A$ is a monotone Carath\'eodory vector valued function defined on $W^{1,p}_0(\Omega)$ for $p>1$ and the $p$-capacity uniform thickness condition is imposed on the complement of our bounded domain $\Omega$. Moreover, for given data $F \in L^p(\Omega;\mathbb{R}^n)$, the problem is set up  with general Dirichlet boundary data $\sigma \in W^{1-1/p,p}(\partial\Omega)$. In this paper, the optimal good-$\lambda$ type bounds technique is applied to prove some results of fractional maximal estimates for gradient of solutions.  And the main ingredients are the action of the cut-off fractional maximal functions and some local interior and boundary comparison estimates developed in previous works \cite{55QH4, MPT2018, MPT2019} and references therein.

\medskip

\medskip

\medskip

\noindent 

\medskip


\end{abstract}   
                  
\section{Introduction and statement of main results}
\label{sec:intro}

The main results of this paper is the gradient estimates of solutions to quasilinear elliptic equations coupled with nonhomogeneous Dirichlet boundary conditions of the form:
\begin{align}
\label{eq:diveq}
\begin{cases}
\div(A(x,\nabla u)) &= \ \div(|F|^{p-2}F) \quad \text{in} \ \ \Omega, \\
\hspace{1.2cm} u &=\ \sigma \qquad \qquad \qquad \text{on} \ \ \partial \Omega.
\end{cases}
\end{align}
Moreover, this study also provides the new fractional maximal estimates for gradients of solutions to this type of problem. Here, the domain $\Omega$ is open bounded domain of $\mathbb{R}^n$, $(n \ge 2)$ and functional data $F \in L^p(\Omega;\mathbb{R}^n)$ together with the general Dirichlet boundary data $\sigma \in W^{1-1/p,p}(\partial\Omega)$. 

Over the past decades, the earlier works on the local interior gradient estimates of weak solutions for classical homogeneous $p$-Laplacian equations:
\begin{align}
\label{eq:plaplace}
\div(|\nabla u|^{p-2}\nabla u)=0
\end{align}
in the scalar case $n=1$ have developed in the series of papers \cite{Ural1968, LadyUral1968, Uhlenbeck1977,Evans1982} for $p \ge 2$, and in \cite{Lewis1983} the same result for the case  $1<p<2$. Later, in \cite{Tolksdorf1984}, the result was extended for more general equation $\div(|\nabla u|^{p-2}\nabla u) = f$ with $p \ge 2$. It is well known that when $p \neq 2$, weak solution to \eqref{eq:plaplace} is of class $C^{1,\alpha}$ that has H\"older continuous derivatives. Besides that, for related results concerning to this equation, in \cite{DiBenedetto1983, Lieberman1984, Lieberman1986, Lieberman1988}, authors proved interior $C^{1,\alpha}$ regularity for homogeneous quasilinear elliptic equations of type $-\div(A(x,u,\nabla u))=0$. Since then, the regularity theory of divergence form modeled on the $p$-Laplacian equation $\div(|\nabla u|^{p-2}\nabla u)=-\div(|F|^{p-2}F)$ (homogeneous problem) has been continued to extend via literature in \cite{Iwaniec, DiBenedetto1993} and many references therein. 

In recent years, there have been plenty of research activities on the regularity theory of solution to quasilinear elliptic equations $\div(A(x,\nabla u)) = \ \div(|F|^{p-2}F)$, under various assumptions on domain, nonlinear operator $A$ and boundary data. For instance, with homogeneous Dirichlet problem (zero Dirichlet boundary data), S. Byun et al. in \cite{SSB3, SSB4, SSB1, SSB2} have been studied gradient estimates of solution in the setting of classical Lebesgue spaces (see \cite{SSB1}) and weighted Lebesgue spaces (see \cite{SSB2}) under assumptions on Reifenberg domain $\Omega$, together with standard ellipticity condition of $A$ and small BMO oscillation in $x$. Under some assumptions of $A$ and the smoothness requirement on domain $\Omega$, further generalization to this type of homogeneous equation are the subjects of \cite{CM2014, Phuc2015, CoMi2016, BCDKS, BW1, KZ} and their related references. 

Afterwards, more general extensions of regularity to the non-homogeneous quasilinear elliptic equations of type $\div(A(x,\nabla u)) = \div F$ was discussed and addressed in many papers. In particular, the interior $W^{1,q}$ estimates was investigated in \cite{Truyen2016} using the perturbations method proposed by Caffarelli et al. in \cite{CP1998}. And recently further, T. Nguyen in \cite{Truyen2018} proved the global gradient estimates in weighted Morrey spaces for solutions where the nonlinearity $A(x,\xi)$ is measurable in $x$, differentiable in $\xi$ and satisfies a small BMO condition, domain $\Omega$ satisfies the Reifenberg flat condition. And many papers related to the same topic could be found in \cite{SSB1, SSB2, SSB3, SSB4, BW1, Tuoc2018}, but under different hypotheses on domain $\Omega$, the nonlinearity $A$ and the given boundary data.

The technique using maximal functions was first presented by G. Mingione et al. in their fine papers \cite{Duzamin2,55DuzaMing} with a nonlinear potential theory, and later, this approach has been developed in many optimal regularity results. In this context, we follow and continue this topic of gradient estimates for the problem of divergence type with general Dirichlet boundary data \eqref{eq:diveq}. Our main results are estimates for gradient of solutions in Lorentz spaces and moreover, some tricks involving the cut-off fractional maximal function are used to give a proof of \emph{fractional maximal gradient estimates},  will be also found in our work. More specifically and precisely, in our study, we only need the assumption of domain $\Omega$ whose complement satisfies $p$-capacity uniform thickness condition, the weaker condition on $\Omega$ than Reifenberg flatness. One notices that this $p$-capacity density condition is stronger than the Weiner criterion described in~\cite{Kilp} as:
\begin{align*}
\int_0^1{\left(\frac{\text{cap}_p((\mathbb{R}^n \setminus \Omega) \cap \overline{B}_r(x), B_{2r}(x))}{\text{cap}_p(\overline{B}_r(x), B_{2r}(x))} \right)^{\frac{1}{p-1}} \frac{dr}{r}} = \infty,
\end{align*} 
which characterizes regular boundary points for the $p$-Laplace Dirichlet problem, where one measures the thickness of complement of $\Omega$ near boundary by capacity densities. The class of domains whose complement satisfies the uniformly $p$-capacity condition is relatively large (including those with Lipschitz boundaries or satisfy a uniform corkscrew condition), and its definition will be highlighted in Section \ref{sec:pcapa}. Otherwise, it is weaker than the Reifenberg flatness condition that was discussed in various studies \cite{BP14, BR, BW1_1, BW2, MP11, MP12,  ER60}. Additionally, the nonlinearity $A$ here is a Carath\'edory vector valued function defined on $W^{1,p}_0(\Omega)$ only satisfying the growth and monotonicity conditions: there holds
\begin{align*}
\left| A(x,\xi) \right| &\le \Lambda_1 |\xi|^{p-1},\\
\langle A(x,\xi)-A(x,\eta), \xi - \eta \rangle &\ge \Lambda_2 \left( |\xi|^2 + |\eta|^2 \right)^{\frac{p-2}{2}}|\xi - \eta|^2,
\end{align*}
for every $(\xi,\eta) \in \mathbb{R}^n \times \mathbb{R}^n \setminus \{(0,0)\}$ and a.e. $x \in \mathbb{R}^n$, $\Lambda_1$ and $\Lambda_2$ are positive constants. This operator and its properties are emphasized in Section \ref{sec:A}. The proofs of our studies are based on the method developed in \cite{55QH4,MPT2018,MPT2019} for the measure data problem, so-called the ``good-$\lambda$'' technique. Our results in this paper show that strong proofs with less hypotheses, to more general problem than previous studies, where our technique of the optimal good-$\lambda$ method (see \cite{MPT2018, MPT2019}) is applied to problem with functional data instead of the measure data.

Now let us give a precise statement of our main results. We firstly give the boundedness property of maximal function via the following theorem \ref{theo:maintheo_lambda}, that will be important for us to prove our results later. 

\begin{theorem}
\label{theo:maintheo_lambda}
	Let $p>1$ and suppose that $\Omega \subset \mathbb{R}^n$ is a bounded domain whose complement satisfies a $p$-capacity uniform thickness condition with constants $c_0, r_0>0$. 
	Then, for  any solution $u$ to \eqref{eq:diveq} with given data $F$, there exist $a \in (0,1)$, $b>0$, $\varepsilon_0 = \varepsilon_0(n,a,b) \in (0,1)$ and a constant $C = C(n,p,\Lambda_1,\Lambda_2,c_0,diam(\Omega)/{r_0})>0$ such that the following estimate
	\begin{align}
	\label{eq:mainlambda}
	\begin{split}
	&\mathcal{L}^n\left(\{{\mathbf{M}}(|\nabla u|^p)>\varepsilon^{-a}\lambda, {\mathbf{M}}(|F|^p+|\nabla \sigma|^p) \le \varepsilon^b\lambda \}\cap \Omega \right)\\ &~~~~~~\qquad \leq C \varepsilon \mathcal{L}^n\left(\{ {\mathbf{M}}(|\nabla u|^p)> \lambda\}\cap \Omega \right),
	\end{split}
	\end{align}
	holds for any $\lambda>0$ and $\varepsilon \in (0,\varepsilon_0)$. Here, we note that $a$ and $b$ are the parameters depending only on $n,p,\Lambda_1, \Lambda_2,c_0$, and will be clarified in our proof later.
\end{theorem}

 Throughout the paper, the denotation $diam(\Omega)$ is the diameter of a set $\Omega$ defined as:
\begin{align*}
diam(\Omega) = \sup\{d(x,y) \  : \ x,y \in \Omega\},
\end{align*}
and the notation $\mathcal{L}^n(E)$ stands for the $n$-dimensional Lebesgue measure of a set $E \subset \mathbb{R}^n$. Moreover, it can be noticed that in this theorem and in what follows, for simplicity, the set $\{x \in \Omega: |g(x)| > \Lambda\}$ is denoted by $\{|g|>\Lambda\}$ (in order to avoid the confusion that may arise). 

With this regard, our first result concerns gradient norm estimates in classical Lorentz spaces.
\begin{theorem}
\label{theo:regularityM0}
Let $p>1$ and $\Omega \subset \mathbb{R}^n$ be a bounded domain whose complement satisfies a $p$-capacity uniform thickness condition with constants $c_0, r_0>0$. Then, for any solution $u$ to \eqref{eq:diveq} with given functional data $F \in W^{1,p}(\Omega)$, $\sigma \in L^p(\Omega)$, $0<q<\frac{\Theta}{p}$ and $0<s \le \infty$, there exists a constant $C = C(n,p,\Lambda_1,\Lambda_2,c_0,r_0,diam(\Omega),q,s)$ such that the following inequality holds
\begin{align}
\label{eq:regularityM0}
\|\mathbf{M}(|\nabla u|^p)\|_{L^{q,s}(\Omega)}\leq C \|\mathbf{M}(|F|^p+|\nabla \sigma|^p)\|_{L^{q,s}(\Omega)}.
\end{align}
\end{theorem}

Next, we state Theorem \ref{theo:main2} in a somewhat more general form of Theorem \ref{theo:maintheo_lambda} as following, where its proof would be found in Section \ref{sec:proofs}. 

\begin{theorem}
\label{theo:main2}
Let $p>1$, $0\le\alpha<n$ and suppose that $\Omega \subset \mathbb{R}^n$ is a bounded domain whose complement satisfies a $p$-capacity uniform thickness condition with constants $c_0, r_0>0$. Then for any solution $u$ to equation~\eqref{eq:diveq} with given data function $F$, there exist  $a \in (0,1)$, $b>0$, $\varepsilon_0 = \varepsilon_0(n,a,b) \in (0,1)$ and a constant $C = C(n, p, \Lambda_1, \Lambda_2,  \alpha, c_0, r_0, T_0= diam(\Omega))$ such that the following estimate
\begin{align*}
\mathcal{L}^n\left(V_{\lambda}^{\alpha}\right) \le C \varepsilon \mathcal{L}^n\left(W_{\lambda}^{\alpha}\right),
\end{align*}
holds for any $\lambda>0$, $\varepsilon \in (0,\varepsilon_0)$ small enough, where
$$V_{\lambda}^{\alpha} = \left\{{\mathbf{M}} {\mathbf{M}}_{\alpha}(|\nabla u|^p) >\varepsilon^{-a} \lambda, \ {\mathbf{M}}_{\alpha}(|F|^p + |\nabla\sigma|^p) \le \varepsilon^{b} \lambda \right\} \cap \Omega,$$
and
$$W_{\lambda}^{\alpha} = \left\{{\mathbf{M}} {\mathbf{M}}_{\alpha} (|\nabla u|^p) > \lambda\right\} \cap \Omega.$$
Note that, as we will discuss on, $a$ and $b$ are the parameters depending only on $n,p,\Lambda_1, \Lambda_2,c_0$, and will be clarified in our proof later.
\end{theorem}

The next theorem shows the improvement of Lorentz gradient estimate in previous theorem \ref{theo:regularityM0}. It concludes the fractional gradient estimate of solutions to our class of nonhomegeneous equations \eqref{eq:diveq} with respect to our data $F$ and $\sigma$. 

\begin{theorem}
\label{theo:regularityMalpha}
Let $p>1$, $0\le\alpha<n$ and $\Omega \subset \mathbb{R}^n$ be a bounded domain whose complement satisfies a $p$-capacity uniform thickness condition with constants $c_0, r_0>0$. Then, for any solution $u$ to \eqref{eq:diveq} with given functional data $F \in W^{1,p}(\Omega)$, $\sigma \in L^p(\Omega)$, $0<q<\frac{\Theta n}{p(n-\alpha)}$ and $0<s \le \infty$, the following inequality
\begin{align}
\label{eq:regularityMalpha}
\|{\mathbf{M}}_\alpha(|\nabla u|^p)\|_{L^{q,s}(\Omega)}\leq C \|\mathbf{M}_\alpha(|F|^p+|\nabla \sigma|^p)\|_{L^{q,s}(\Omega)}
\end{align}
holds. Here, the constant $C$ depending only on $n,p,\Lambda_1,\Lambda_2,\alpha,c_0,r_0,diam(\Omega),q,s$. 
\end{theorem}

For the proofs of these above Theorems in our present paper, it is possible to apply some gradient estimate results developed for quasilinear equations with measure data, or linear/nonlinear potential and Calder\'on-Zygmund theories (see \cite{DuMin2010, DuMin2011, Mi3, Mi2019, MPT2018, MPT2019}).

Here, the results of Theorems \ref{theo:main2} and \ref{theo:regularityMalpha} generalize that of Theorems \ref{theo:maintheo_lambda} and \ref{theo:regularityM0} above. Heuristically speaking, one can expect such a stronger result with general fractional maximal gradient estimates of solutions ($0 \le \alpha<n$), where $\alpha=0$ is just a specific case. However, the proofs are generalized naturally in a different approach via cut-off fractional maximal functions and their related properties, somewhat will be described in this survey paper. The case study remains an open problem that leads us to consider a small range of $\alpha$ (for $0 \le \alpha<1$ and $\alpha$ is very closed to 0). This would be more meaningful for us because then, we can enlarge range of $q$ in results of Theorems \ref{theo:regularityM0} and \ref{theo:regularityMalpha}. Of course in that context, one should obtain more regularities with the fewest and simplest assumptions on $\Omega$ and the operator $A$ (in some previous works by others, the Reifenberg flatness of $\Omega$ and a small BMO condition was added to $A$). That seems interesting to many researchers and may appear in the topic of an upcoming work.

The outline of this paper is organized as follows. In the next Section \ref{sec:preliminaries} we begin with some preliminaries, gather few notations and assumptions on the problem, which are useful to our proofs. The next Section \ref{sec:cutoff} is devoted to study the cut-off fractional maximal functions and proofs of some preparatory lemmas related to them are obtained therein. To make effective use of the good-$\lambda$ method, Section \ref{sec:inter_bound} indicates to some important lemmas of local interior and boundary comparison estimates and then finally, the proofs of main theorems are provided in the last Section \ref{sec:proofs}.

\section{Preliminaries}
\label{sec:preliminaries}

This section provides some necessary preliminaries and  we also recall some well-known notations and results for later use. 

\subsection{The uniform $p$-capacity condition}
\label{sec:pcapa}
In this paper, the considered domain $\Omega \subset \mathbb{R}^n$ is under the assumption that its complement $\mathbb{R}^n \setminus \Omega$ is uniformly $p$-capacity thick. More precise, we say that the domain $\mathbb{R}^n \setminus \Omega$ satisfies the \emph{$p$-capacity uniform thickness condition} if there exist two constants $c_0,r_0>0$ such that
\begin{align}
\label{eq:capuni}
\text{cap}_p((\mathbb{R}^n \setminus \Omega) \cap \overline{B}_r(x), B_{2r}(x)) \ge c_0 \text{cap}_p(\overline{B}_r(x),B_{2r}(x)),
\end{align}
for every $x \in \mathbb{R}^n \setminus \Omega$ and $0<r \le r_0$. Here, the $p$-capacity of any compact set $K \subset \Omega$ (relative to $\Omega$) is defined as:
\begin{align*}
\text{cap}_p(K,\Omega) = \inf \left\{ \int_\Omega{|\nabla \varphi|^p dx}: \varphi \in C_c^\infty, \varphi \ge \chi_K \right\},
\end{align*}
where $\chi_K$ is the characteristic function of $K$. The $p$-capacity of any open subset $U \subseteq \Omega$ is then defined by:
\begin{align*}
\text{cap}_p(U,\Omega) = \sup \left\{ \text{cap}_p(K,\Omega), \ K \ \text{compact}, \ K \subseteq U \right\}.
\end{align*}
Consequently, the $p$-capacity of any subset $B \subseteq \Omega$ is defined by:
\begin{align*}
\text{cap}_p(B,\Omega) = \inf \left\{ \text{cap}_p(U,\Omega), \ U \ \text{open}, \ B \subseteq U \right\}.
\end{align*}
A function $u$ defined on $\Omega$ is said to be $\text{cap}_p$-quasi continuous if for every $\varepsilon>0$ there exists $B \subseteq \Omega$ with $\text{cap}_p(B,\Omega) < \varepsilon$ such that the restriction of $u$ to $\Omega\setminus B$ is continuous. Every nonempty $\mathbb{R}^n\setminus \Omega$ is uniform $p$-thick for $p >n$ and this condition is nontrivial only when $p \le n$. For some properties of the $p$-capacity we refer to \cite{HKM2006}.

\subsection{Assumptions on operator $A$}
\label{sec:A}
In our study of elliptic equations $\div(A(x,\nabla u)) = \ \div(|F|^{p-2}F)$, the nonlinear operator $A: \Omega \times \mathbb{R}^n \rightarrow \mathbb{R}$ is a Carath\'eodory vector valued function (that is, $A(.,\xi)$ is measurable on $\Omega$ for every $\xi$ in $\mathbb{R}^n$, and $A(x,.)$ is continuous on $\mathbb{R}^n$ for almost every $x$ in $\Omega$) which satisfies the following growth and monotonicity conditions: for some $1<p\le n$ there exist two positive constants $\Lambda_1$ and $\Lambda_2$ such that 
\begin{align}
\label{eq:A1}
\left| A(x,\xi) \right| &\le \Lambda_1 |\xi|^{p-1},
\end{align}
and
\begin{align}
\label{eq:A2}
\langle A(x,\xi)-A(x,\eta), \xi - \eta \rangle &\ge \Lambda_2 \left( |\xi|^2 + |\eta|^2 \right)^{\frac{p-2}{2}}|\xi - \eta|^2
\end{align}
holds for almost every $x$ in $\Omega$ and every $(\xi,\eta) \in \mathbb{R}^n \times \mathbb{R}^n \setminus \{(0,0)\}$.

\subsection{Lorentz spaces}
\label{sec:lorentz}
Let us firstly recall the definition of the \emph{Lorentz space} $L^{q,t}(\Omega)$ for $0<q<\infty$ and $0<t\le \infty$ (see in \cite{55Gra}). It is the set of all Lebesgue measurable functions $g$ on $\Omega$ such that:
\begin{align}
\label{eq:lorentz}
\|g\|_{L^{q,t}(\Omega)} = \left[ q \int_0^\infty{ \lambda^t\mathcal{L}^n \left( \{x \in \Omega: |g(x)|>\lambda\} \right)^{\frac{t}{q}} \frac{d\lambda}{\lambda}} \right]^{\frac{1}{t}} < +\infty,
\end{align}
as $t \neq \infty$. If $t = \infty$, the space $L^{q,t}(\Omega)$ is the usual weak-$L^q$ or Marcinkiewicz spaces with the following quasinorm:
\begin{align}
\|g\|_{L^{q,\infty}(\Omega)} = \sup_{\lambda>0}{\lambda \mathcal{L}^n\left(\{x \in \Omega:|g(x)|>\lambda\}\right)^{\frac{1}{q}}}.
\end{align}
When $t=q$, the Lorentz space $L^{q,q}(\Omega)$ becomes the Lebesgue space $L^q(\Omega)$. 

\subsection{Maximal and Fractional Maximal functions}
\label{sec:MMalpha}

In what follows, we denote the open ball in $\mathbb{R}^n$ with center $x_0$ and radius $r$ by $B_r(x_0)$, that is the set $B_r(x_0) = \{x\in \mathbb{R}^n: |x-x_0|<r\}$. And we clarify that in this paper, we use the denotation $\displaystyle{\fint_{B_r(x)}{f(y)dy}}$ indicates the integral average of $f$ in the variable $y$ over the ball $B_r(x)$, i.e.
\begin{align*}
\fint_{B_r(x)}{f(y)dy} = \frac{1}{|B_\rho(x)|}\int_{B_r(x)}{f(y)dy}.
\end{align*}

We first recall the definition of fractional maximal function that regarding to \cite{K1997, KS2003}. Let $0 \le \alpha \le n$, the fractional  maximal function $\mathbf{M}_\alpha$ of a locally integrable function $g: \mathbb{R}^n \rightarrow [-\infty,\infty]$ is defined by:
\begin{align}
\label{eq:Malpha}
\mathbf{M}_\alpha g(x) = \sup_{\rho>0}{\rho^\alpha \fint_{B_\rho(x)}{|g(y)|dy}}.
\end{align}
For the case $\alpha=0$, one obtains the Hardy-Littlewood maximal function, $\mathbf{M}g = \mathbf{M}_0g$, defined for each locally integrable function $g$ in $\mathbb{R}^n$ by:
\begin{align}
\label{eq:M0}
\mathbf{M}g(x) = \sup_{\rho>0}{\fint_{B_{\rho}(x)}|g(y)|dy},~~ \forall x \in \mathbb{R}^n.
\end{align}
Fractional maximal operators have may applications in partial differential equations, potential theory and harmonic analysis. Once we need to estimate some quantities of a function $g$, they can be shown to be dominated by $\mathbf{M}g$, or more generally by $\mathbb{M}_\alpha g$. The fundamental result of maximal operator is that the boundedness on $L^p(\mathbb{R}^n)$ when $1< p \le \infty$, that is there exists a constant $C(n,p)>0$ such that:
\begin{align*}
\|\mathbf{M}g\|_{L^p(\mathbb{R}^n)} \le C(n,p)\|g\|_{L^p(\mathbb{R}^n)}, \quad \forall g \in L^p(\mathbb{R}^n).
\end{align*}
Moreover, $\mathbf{M}$ is also said to be \emph{weak- type} (1,1), this means there is a constant $C(n)>0$  such that for all $\lambda>0$ and $g \in L^1(\mathbb{R}^n)$, it holds that
\begin{align*}
\mathcal{L}^n\left(\{\mathbf{M}(g)>\lambda\}\right) \le C(n)\frac{\|g\|_1}{\lambda}.
\end{align*}

The standard and classical references can be found in many places such as \cite{Gra97,55Gra}, and later also in \cite{55QH10}. Besides that, there are some well-known properties of maximal and fractional maximal operators, that will be shown in following lemmas.

\begin{lemma}
\label{lem:boundM}
It refers to \cite{55Gra} that the operator $\mathbf{M}$ is bounded from $L^s(\mathbb{R}^n)$ to $L^{s,\infty}(\mathbb{R}^n)$, for $s \ge 1$, this means,
\begin{align}
\mathcal{L}^n\left(\{\mathbf{M}(g)>\lambda\}\right) \le \frac{C}{\lambda^s}\int_{\mathbb{R}^n}{|g(x)|^s dx}, \quad \mbox{ for all } \lambda>0.
\end{align}
\end{lemma}
\begin{lemma}
\label{lem:boundMlorentz}
In \cite{55Gra}, it allows us to present a boundedness property of maximal function $\mathbf{M}$ in the Lorentz space $L^{q,s}(\mathbb{R}^n)$, for $q>1$ as follows:
\begin{align}
\label{eq:boundM}
\|\mathbf{M}g\|_{L^{q,s}(\Omega)} \le C \|g\|_{L^{q,s}(\Omega)}.
\end{align}
\end{lemma}

Moreover, a very important property of fractional maximal function was also obtained from the boundedness property of maximal function. The proof of this result is a modification of the result in Lemma \ref{lem:boundM} based on the definition of maximal and fractional maximal function, and we show below all the details.
\begin{lemma}
\label{lem:Malpha_prop}
Let $0\le \alpha<n$, $\rho>0$ and $x \in \mathbb{R}^n$. Then, for any locally integrable function $f \in L^1_{\text{loc}}(\mathbb{R}^n)$ we have the following inequality holds
\begin{equation*}
\displaystyle{\mathcal{L}^n\left(\left\{\mathbf{M}_\alpha(\chi_{B_{\rho}(x)}f)>\lambda\right\}\right)} \leq C\displaystyle{\left(\frac{\displaystyle{\int_{B_{\rho}(x)}|f(y)|dy}}{\lambda}\right)^{\frac{n}{n-\alpha}}},~~ \mbox{ for all} \ \lambda>0.
\end{equation*}
\end{lemma}
\begin{proof}
First of all, let us give a proof that for $0 \le \alpha<n$ and any $f \in L^1_{\text{loc}}(\mathbb{R}^n)$, there holds:
\begin{equation*}
\displaystyle{\mathcal{L}^n\left(\left\{\mathbf{M}_\alpha f>1\right\}\right)} \leq C\displaystyle{\left({\displaystyle{\int_{\mathbb{R}^n}|f(y)|dy}}\right)^{\frac{n}{n-\alpha}}}.
\end{equation*}
Indeed, for any $x \in \mathbb{R}^n$, from definition of fractional maximal function $\textbf{M}_\alpha$, one has
\begin{align*}
\mathbf{M}_{\alpha} f(x) &= \sup_{\rho>0} \rho^{\alpha-n} \int_{B_{\rho}(x)}|f(y)|dy \\
& = \sup_{\rho>0} \left(\rho^{-n} \int_{B_{\rho}(x)}|f(y)|dy\right)^{\frac{n-\alpha}{n}} \left(\int_{B_{\rho}(x)}|f(y)|dy\right)^{\frac{\alpha}{n}} \\
& \le C \left[\mathbf{M}f(x)\right]^{1 - \frac{\alpha}{n}} \|f\|^{\frac{\alpha}{n}}_{L^1(\mathbb{R}^n)}.
\end{align*}
It follows that
\begin{align*}
\displaystyle{\mathcal{L}^n\left(\left\{\mathbf{M}_\alpha f>1\right\}\right)}  &\leq \displaystyle{\mathcal{L}^n\left(\left\{[\mathbf{M} f]^{1 - \frac{\alpha}{n}} \|f\|^{\frac{\alpha}{n}}_{L^1(\mathbb{R}^n)}>C\right\}\right)} \\ & = \displaystyle{\mathcal{L}^n\left(\left\{\mathbf{M} f > C \|f\|^{-\frac{\alpha}{n-\alpha}}_{L^1(\mathbb{R}^n)}\right\}\right)} .
\end{align*}
Applying Lemma~\ref{lem:boundM} for $s=1$ and $\lambda = \|f\|^{-\frac{\alpha}{n-\alpha}}_{L^1(\mathbb{R}^n)}$, we obtain that
\begin{align*}
\displaystyle{\mathcal{L}^n\left(\left\{\mathbf{M}_\alpha f>1\right\}\right)}  \le \frac{C}{ \|f\|^{-\frac{\alpha}{n-\alpha}}_{L^1(\mathbb{R}^n)}}  \int_{\mathbb{R}^n} |f(x)|dx  = C  \|f\|^{\frac{n}{n-\alpha}}_{L^1(\mathbb{R}^n)}.
\end{align*}
Without loss of generality,  by scaling what already proved, we consider function $\frac{f}{\lambda}$ instead of $f \in L^1_{\text{loc}}(\mathbb{R}^n)$ and then with $\lambda$ is 1, it yields that the following inequality holds
\begin{equation*}
\displaystyle{\mathcal{L}^n\left(\left\{\mathbf{M}_\alpha(\chi_{B_{\rho}(x)}f)>\lambda\right\}\right)} \leq C\displaystyle{\left(\frac{\displaystyle{\int_{B_{\rho}(x)}|f(y)|dy}}{\lambda}\right)^{\frac{n}{n-\alpha}}},~~ \mbox{for all} \ \lambda>0.
\end{equation*}
\end{proof}


\section{Cut-off Fractional Maximal functions and Preparatory lemmas}
\label{sec:cutoff}

In this section, we restrict ourselves to study the so-called ``cut-off fractional maximal functions'' and their properties that will be needed in later parts of this paper.

Let $r>0$ and $0\le \alpha \le n$, we define some additional cut-off maximal functions of a locally integrable function $f$ corresponding to the maximal function $\mathbf{M}f$ in \eqref{eq:M0} as follows
\begin{align}
\label{eq:MTrf}
\begin{split}
{\mathbf{M}}^rf(x)  &= \sup_{0<\rho<r} \fint_{B_\rho(x)}f(y)dy; \\ {\mathbf{T}}^rf(x) &= \sup_{\rho \ge r}\fint_{B_\rho(x)}f(y)dy,
\end{split}
\end{align}
and corresponding to $\mathbf{M}_{\alpha}f$ in \eqref{eq:Malpha} as
\begin{align}
\label{eq:MTralphaf}
{\mathbf{M}}^r_{\alpha}f(x)  &= \sup_{0<\rho<r} \rho^{\alpha} \fint_{B_\rho(x)}f(y)dy; \\ {\mathbf{T}}^r_{\alpha}f(x) &= \sup_{\rho \ge r} \rho^{\alpha}\fint_{B_\rho(x)}f(y)dy.
\end{align}

We remark here that if $\alpha = 0$ then $\mathbf{M}^r_{\alpha}f = \mathbf{M}^rf$ and $\mathbf{T}^r_{\alpha}f = \mathbf{T}^r f$, for all $f \in L^1_{loc}(\mathbb{R}^n)$. The following lemma can be inferred from from their definitions.

\begin{lemma}\label{lem:Malpha}
For any $r>0$ and $0\le  \alpha \le n$, we have
\begin{align*}
{\mathbf{M}}{\mathbf{M}}_{\alpha}f(x) \le \max \left\{{\mathbf{M}}^r{\mathbf{M}}^r_{\alpha}f(x), {\mathbf{M}}^r{\mathbf{T}}^r_{\alpha}f(x), {\mathbf{T}}^r{\mathbf{M}}_{\alpha}f(x)\right\},
\end{align*}
for any $x \in \mathbb{R}^n$ and $f \in L^1_{loc}(\mathbb{R}^n)$.
\end{lemma}

We will now prove some inequalities related to these operators, that will be needed in our desired results later.

\begin{lemma}\label{lem:Tr}
Let $r >0$, $k \ge 1$ and $0\le \alpha \le n$. For some $x_1, x_2 \in \mathbb{R}^n$, assume that 
$$B_{\rho}(x_1) \subset B_{k\rho}(x_2), \quad \forall \rho \ge r.$$ 
Then we have the following estimate
\begin{align}
{\mathbf{T}}^r_{\alpha}f(x_1)  \le k^{n-\alpha} {\mathbf{M}}_{\alpha}f(x_2),
\end{align}
for all $f \in L^1_{loc}(\mathbb{R}^n)$.
\end{lemma}
\begin{proof}
From definition of the cut-off $\mathbf{T}_\alpha^r$ in \eqref{eq:MTralphaf}, the inequality is proved as well:
\begin{align*}
{\mathbf{T}}^r_{\alpha}f(x_1) & = \sup_{\rho\ge r} \rho^{\alpha-n}  \int_{B_{\rho}(x_1)} f(x)dy \\
&\le \sup_{\rho\ge r} \rho^{\alpha-n} \int_{B_{k\rho}(x_2)} f(x)dy \\ 
& = k^{n - \alpha} \sup_{\rho\ge r} \ (k\rho)^{\alpha} \fint_{B_{k\rho}(x_2)} f(x)dy\\
& \le  k^{n - \alpha} {\mathbf{M}}_{\alpha}f(x_2).
\end{align*}
\end{proof}

\begin{lemma}\label{lem:MrMr}
Let $r>0$ and $0\le \alpha <n$. Then there exists a constant $C>0$ such that
\begin{align}\label{eq:MrMr}
{\mathbf{M}}^r {\mathbf{M}}^r_{\alpha} f(x) \le C {\mathbf{M}}^{2r}_{\alpha}f(x),
\end{align}
for any $x \in \mathbb{R}^n$ and $f \in L^1_{loc}(\mathbb{R}^n)$.
\end{lemma}
\begin{proof}
For any $\rho \in (0,r)$ and $y \in B_{\rho}(x)$, we have
\begin{align}\label{eq:Mr1}
{\mathbf{M}}^r_{\alpha} f(y)  = \max\left\{ {\mathbf{M}}^{\rho}_{\alpha} f(y), \sup_{\rho\le\delta<r} \delta^{\alpha-n} \int_{B_{\delta}(y)}f(z) dz \right\}. 
\end{align}
For any $\delta>0$, since $B_{\delta}(y) \subset B_{\rho + \delta}(x)$, it deduces that the second term on the right-hand side can be estimated as
\begin{align}\nonumber
\sup_{\rho\le\delta<r} \delta^{\alpha-n} \int_{B_{\delta}(y)}f(z) dz & \le \sup_{\rho\le\delta<r} \left(\frac{\delta}{\rho+\delta}\right)^{\alpha-n} (\rho+\delta)^{\alpha-n} \int_{B_{\rho+\delta}(x)}f(z) dz \\ \nonumber
& \le 2^{n-\alpha} \sup_{\rho\le\delta<r} (\rho+\delta)^{\alpha-n} \int_{B_{\rho+\delta}(x)}f(z) dz \\ \nonumber
& \le 2^{n-\alpha} \sup_{0<R<2r} R^{\alpha-n} \int_{B_{R}(x)}f(z) dz \\
\label{eq:Mr2}
& = 2^{n-\alpha} {\mathbf{M}}^{2r}_{\alpha}f(x). 
\end{align}
From \eqref{eq:Mr1}, \eqref{eq:Mr2} and the definitions of the cut off fractional maximal function ${\mathbf{M}}^r$ and ${\mathbf{M}}^r_{\alpha}$ in \eqref{eq:MTrf} and \eqref{eq:MTralphaf}, one obtains
\begin{align}\nonumber
{\mathbf{M}}^r {\mathbf{M}}^r_{\alpha} f(x) & = \sup_{0<\rho<r} \rho^{-n} \int_{B_{\rho}(x)} {\mathbf{M}}^r_{\alpha} f(y) dy \\ \nonumber
&\le \max\left\{\sup_{0<\rho<r} \rho^{-n} \int_{B_{\rho}(x)} {\mathbf{M}}^{\rho}_{\alpha} f(y) dy, \ 2^{n-\alpha} {\mathbf{M}}^{2r}_{\alpha}f(x)\right\} \\ \label{eq:Mr3}
& = \max\left\{I, \  2^{n-\alpha} {\mathbf{M}}^{2r}_{\alpha}f(x)\right\},
\end{align}
where
\begin{align*}
I & = \sup_{0<\rho<r} \rho^{-n} \int_{B_{\rho}(x)} {\mathbf{M}}^{\rho}_{\alpha} f(y) dy,
\end{align*}
and it remains to prove the estimate for this term. Indeed, for any $y\in B_{\rho}(x)$, we have
\begin{equation*}
{\mathbf{M}}^{\rho}_{\alpha} f(y) = \sup_{0<\delta<\rho} \delta^{\alpha-n} \int_{B_{\delta}(y)} \chi_{B_{2\rho}(x)}f(z)dz\leq \mathbf{M}_\alpha[ \chi_{B_{2\rho}(x)}f](y).
\end{equation*}
and it clearly forces that
\begin{align*}
I\leq   \sup_{0<\rho<r} \rho^{-n} \int_{B_{\rho}(x)} \mathbf{M}_\alpha[\chi_{B_{2\rho}(x)}f](y) dy.
\end{align*}
According to the results in Lemma \ref{lem:Malpha_prop}, we thus get
\begin{align*}
 \int_{B_{\rho}(x)} \mathbf{M}_\alpha[\chi_{B_{2\rho}(x)}f](y) dy&= \int_{0}^{\infty} \mathcal{L}^n\left(\left\{\mathbf{M}_\alpha(\chi_{B_{2\rho}(x)}f)>\lambda\right\}\right) d\lambda\\& \leq  C\rho^n\lambda_0+\int_{\lambda_0}^{\infty} \mathcal{L}^n\left(\left\{\mathbf{M}_\alpha(\chi_{B_{2\rho}(x)}f)>\lambda\right\}\right) d\lambda\\&\leq C \rho^n\lambda_0+ C \left(\int_{B_{2\rho}(x)}f(y)dy\right)^{\frac{n}{n-\alpha}}\int_{\lambda_0}^{\infty}\lambda^{-\frac{n}{n-\alpha}}d\lambda\\&= C\rho^n\lambda_0+ C \left(\int_{B_{2\rho}(x)}f(y)dy\right)^{\frac{n}{n-\alpha}}\lambda_0^{-\frac{\alpha}{n-\alpha}}.
\end{align*}
By choosing
\begin{equation*}
\lambda_0=\rho^{-n+\alpha}\int_{B_{2\rho}(x)}f(y)dy,
\end{equation*}
we obtain
\begin{align*}
\int_{B_{\rho}(x)} \mathbf{M}_\alpha[\chi_{B_{2\rho}(x)}f](y) dy&\leq C\rho^{\alpha}\int_{B_{2\rho}(x)}f(y)dy.
\end{align*}
Hence, 
\begin{align} \label{eq:Mr4}
I\leq C  \sup_{0<\rho<r} \rho^{-n+\alpha} \int_{B_{2\rho}(x)}f(y)dy \le 2^{n-\alpha} C {\mathbf{M}}^{2r}_{\alpha}f(x).
\end{align}
From \eqref{eq:MrMr} and the easily checked inequalities in \eqref{eq:Mr3} and \eqref{eq:Mr4}, it completes the proof.
\end{proof}

\section{Interior and boundary comparison estimates}
\label{sec:inter_bound}
In this section, we present some local interior and boundary comparison estimates for weak solution $u$ of \eqref{eq:diveq} that are essential to our development later. 

\begin{proposition}
\label{prop1}
Let $\sigma \in W^{1,p}(\Omega), \ F \in L^p(\Omega)$ and $u$ be a weak solution of~\eqref{eq:diveq}. Then we have
\begin{equation}
\label{eq:prop1}
\int_{\Omega} |\nabla u|^p dx \le C \int_{\Omega} \left(|F|^p + |\nabla \sigma|^p \right)dx.
\end{equation}
Here, it remarks that the constant $C$ depends only on $p,\Lambda_1, \Lambda_2$.
\end{proposition}
\begin{proof}
By using $u - \sigma$ as a test function of equation~\eqref{eq:diveq}, we obtain
\begin{equation*}
\int_{\Omega} A(x,\nabla u) \nabla u dx = \int_{\Omega} A(x,\nabla u) \nabla \sigma dx + \int_{\Omega} |F|^{p-2} F \nabla (u - \sigma) dx.
\end{equation*}
It follows from conditions ~\eqref{eq:A1} and~\eqref{eq:A2} of operator $A$ as
\begin{equation*}
\int_{\Omega} |\nabla u|^p dx \le C \left( \int_{\Omega} |\nabla u|^{p-1} |\nabla \sigma| dx + \int_{\Omega} |F|^{p-1} |\nabla u| dx + \int_{\Omega} |F|^{p-1} |\nabla \sigma| dx \right).
\end{equation*}
By using H{\"o}lder's inequality and Young's inequality, we obtain that
\begin{align*}
 \int_{\Omega} |\nabla u|^{p-1} |\nabla \sigma| dx & \le \left( \int_{\Omega} |\nabla u|^{p} dx \right)^{\frac{p-1}{p}} \left( \int_{\Omega} |\nabla \sigma|^p dx \right)^{\frac{1}{p}} \\
 & \le \frac{p-1}{2p} \int_{\Omega} |\nabla u|^{p} dx + \frac{2^{p-1}}{p}  \int_{\Omega} |\nabla \sigma|^{p} dx,
\end{align*}
\begin{align*}
 \int_{\Omega} |F|^{p-1} |\nabla u| dx  \le \frac{1}{2p}  \int_{\Omega} |\nabla u|^{p} dx + \frac{p-1}{p} 2^{\frac{1}{p-1}} \int_{\Omega} |F|^{p} dx ,
\end{align*}
\begin{align*}
 \int_{\Omega} |F|^{p-1} |\nabla \sigma| dx \le \frac{p-1}{p} \int_{\Omega} |F|^{p} dx +  \frac{1}{p}  \int_{\Omega} |\nabla \sigma|^{p} dx.
\end{align*}
We obtain~\eqref{eq:prop1} by combining these estimates.
\end{proof}

\subsection{In the interior domain}

We firstly take our attention to the interior estimates. Let us fix a point $x_0 \in \Omega$, for $0<2R \le r_0$ ($r_0$ was given in \eqref{eq:capuni}). Assume $u$ being solution to \eqref{eq:diveq} and for each ball $B_{2R}=B_{2R}(x_0)\subset\subset\Omega$, we consider the unique solution $w$ to the following equation:
\begin{equation}
\label{eq:I1}
\begin{cases} \mbox{div} A(x,\nabla w) & = \ 0, \quad \ \quad \mbox{ in } B_{2R}(x_0),\\ 
\hspace{1.2cm} w & = \ u - \sigma, \ \mbox{ on } \partial B_{2R}(x_0).\end{cases}
\end{equation}

We first recall the following version of interior Gehring's lemma applied to the function $w$ defined in equation \eqref{eq:I1}, has been studied in \cite[Theorem 6.7]{Giu}. It is also known as a kind of ``reverse'' H\"older inequality with increasing supports. Here, let us mention that the proof of such reserve H\"older type estimates of $\nabla u$ can be found in \cite{Phuc1, MPT2018}. And the use of this inequality with small exponents was firstly proposed by G. Mingione in his fine paper \cite{Mi1} when the problem involves measure data. The reader is referred to \cite{Phuc1, MPT2018, Mi3, 55QH4, HOk2019} and materials therein for the proof of this inequality and related results in similar research papers.

\begin{lemma} 
\label{lem:reverseHolder} 
Let $w$ be the solution to \eqref{eq:I1}. Then, there exist  constants $\Theta = \Theta(n,p,\Lambda_1,\Lambda_2)>p$ and $C = C(n,p,\Lambda_1,\Lambda_2)>0$ such that the following estimate      
	\begin{equation}\label{eq:reverseHolder}
	\left(\fint_{B_{\rho/2}(y)}|\nabla w|^{\Theta} dx\right)^{\frac{1}{\Theta}}\leq C\left(\fint_{B_{\rho}(y)}|\nabla w|^p dx\right)^{\frac{1}{p}}
	\end{equation}
	holds for all  $B_{\rho}(y)\subset B_{2R}(x_0)$. 
\end{lemma}

\begin{lemma}
\label{lem:I1}
Let $w$ be the unique solution to equation~\eqref{eq:I1}. Then, there exists a positive constant $C = C(n,p,\Lambda_1,\Lambda_2)>0$ such that the following comparison estimate
\begin{multline}
\label{eq:lem1b}
\fint_{B_{2R}(x_0)} |\nabla u - \nabla w|^p dx  \le   C \fint_{B_{2R}(x_0)} |F|^p + |\nabla \sigma|^p dx \\ + C \left(\fint_{B_{2R}(x_0)} |\nabla u|^pdx \right)^{\frac{p-1}{p}} \left(\fint_{B_{2R}(x_0)} |\nabla \sigma|^pdx \right)^{\frac{1}{p}}
\end{multline}
holds for all $p>1$.
\end{lemma}
\begin{proof}
By choosing $u - w - \overline{\sigma}$ as a test function of equations~\eqref{eq:diveq} and~\eqref{eq:I1}, where $\overline{\sigma} = \sigma$ in $\overline{B}_{2R}(x_0)$, one can show that
\begin{align}
\label{eq:L1I1}
\nonumber
\int_{B_{2R}(x_0)} \left(A(x,\nabla u) - A(x, \nabla w)\right) \nabla (u - w) dx  = \int_{B_{2R}(x_0)} \left(A(x,\nabla u) - A(x, \nabla w)\right) \nabla \sigma dx \\ + \int_{B_{2R}(x_0)} |F|^{p-2} F \nabla (u - w) dx  - \int_{B_{2R}(x_0)} |F|^{p-2} F \nabla \sigma dx.
\end{align}
Two conditions of the operator $A$ in~\eqref{eq:A1} and~\eqref{eq:A2} immediately yield that there exists a positive constant $C$ depending on $\Lambda_1, \Lambda_2$ such that
\begin{multline}
\label{eq:L1I2}
\int_{B_{2R}(x_0)}  |\nabla u - \nabla w|^p dx \le C \left(\int_{B_{2R}(x_0)} \left(|\nabla u| + |\nabla w |\right)^{p-1} |\nabla \sigma| dx\right.  \\ + \left.  \int_{B_{2R}(x_0)} |F|^{p-1} |\nabla u - \nabla w| dx + \int_{B_{2R}(x_0)} |F|^{p-1}  |\nabla \sigma| dx\right).
\end{multline}
Moreover, let us remark that
\begin{align}
\label{eq:remin}
\begin{split}
\left(|\nabla u| + |\nabla w|\right)^{p-1} & \le \left(2|\nabla u| + |\nabla u - \nabla w|\right)^{p-1} \\ 
&\le 4^p \left(|\nabla u|^{p-1} + |\nabla u - \nabla w|^{p-1}\right).
\end{split}
\end{align}
and follow from~\eqref{eq:L1I2}, it yields
\begin{equation}
\label{eq:L10}
\int_{B_{2R}(x_0)}  |\nabla u - \nabla w|^p dx \le C \left(I_1 + I_2 + I_3 + I_4\right),
\end{equation}
where 
\begin{align*}
& I_1 = \int_{B_{2R}(x_0)} |\nabla u|^{p-1} |\nabla \sigma| dx,\quad I_2= \int_{B_{2R}(x_0)} |\nabla u - \nabla w|^{p-1} |\nabla \sigma| dx,\\
& I_3 = \int_{B_{2R}(x_0)} |F|^{p-1} |\nabla u - \nabla w| dx, \quad \text{and} \quad I_4 = \int_{B_{2R}(x_0)} |F|^{p-1}  |\nabla \sigma| dx.
\end{align*}
For any $\varepsilon>0$, thanks to H\"older's inequality and Young's inequality, it is clearly to obtain the estimations for each term $I_1, I_2$ and $I_3$ as follows:
\begin{equation}
\label{eq:estI0}
I_1  \le   \left(\int_{B_{2R}(x_0)}  |\nabla \sigma|^p dx \right)^{\frac{1}{p}} \left(\int_{B_{2R}(x_0)} |\nabla u|^{p} dx \right)^{\frac{p-1}{p}},
\end{equation}
\begin{align}
\nonumber
I_2 & \le  \left(\int_{B_{2R}(x_0)}  |\nabla \sigma|^p dx \right)^{\frac{1}{p}} \left(\int_{B_{2R}(x_0)} |\nabla u- \nabla w |^{p} dx \right)^{\frac{p-1}{p}} \\
\label{eq:estI1}
& \le \frac{1}{p}{\varepsilon^{1-p}} \int_{B_{2R}(x_0)} |\nabla \sigma|^p dx + \frac{p-1}{p}\varepsilon \int_{B_{2R}(x_0)} |\nabla u- \nabla w |^{p} dx ,
\end{align}
\begin{align}
\nonumber
I_3 & \le  \left(\int_{B_{2R}(x_0)}  |\nabla u - \nabla w|^p dx \right)^{\frac{1}{p}} \left(\int_{B_{2R}(x_0)} |F|^{p} dx \right)^{\frac{p-1}{p}} \\
\label{eq:estI2}
& \le  \frac{1}{p} \varepsilon \int_{B_{2R}(x_0)} |\nabla u - \nabla w|^p dx + \frac{p-1}{p}\varepsilon^{-\frac{1}{p-1}} \int_{B_{2R}(x_0)} |F|^{p} dx,
\end{align}
and
\begin{align}
\nonumber
I_4 & \le  \left(\int_{B_{2R}(x_0)}  |\nabla \sigma|^p dx \right)^{\frac{1}{p}} \left(\int_{B_{2R}(x_0)} |F|^{p} dx \right)^{\frac{p-1}{p}} \\
\label{eq:estI3}
& \le  \frac{1}{p} \int_{B_{2R}(x_0)} |\nabla \sigma|^p dx + \frac{p-1}{p} \int_{B_{2R}(x_0)} |F|^{p} dx.
\end{align}
Choosing $\varepsilon = \frac{1}{2}$ and combining~\eqref{eq:L10} with \eqref{eq:estI0},  \eqref{eq:estI1}, \eqref{eq:estI2} and~\eqref{eq:estI3}, we conclude that \eqref{eq:lem1b} holds, where the constant $C$ depending on $n,p,\Lambda_1, \Lambda_2,c_0$.
\end{proof}
\bigskip

\subsection{On the boundary}

Next, we are able to highlight some comparison estimates on the boundary and the same conclusion as interior estimates can be drawn hereafter. First, as $\mathbb{R}^n \setminus \Omega$ is uniformly $p$-thick with constants $c_0, r_0>0$, let $x_0 \in \partial \Omega$ be a boundary point and for $0<R<r_0/10$ we set $\Omega_{10R} = \Omega_{10R}(x_0) = B_{10R}(x_0) \cap \Omega$. With $u \in W^{1,p}_0(\Omega)$ being a solution to \eqref{eq:diveq}, we consider the unique solution $w \in u+W^{1,p}_0(\Omega_{10R})$ to the following equation:
\begin{equation}
\label{eq:B1}
\left\{ \begin{array}{rcl}
	 \operatorname{div}\left( {A(x,\nabla w)} \right) &=& 0 \quad ~~~~~~\text{in}\quad \Omega_{10R}(x_0), \\ 
w &=& u-\sigma \quad \text{on} \quad \partial \Omega_{10R}(x_0). 
\end{array} \right.
\end{equation}
In what follows we extend $u$ by zero to $\mathbb{R}^n\setminus \Omega$ and $w$ by $u-\sigma$ to $\mathbb{R}^n\setminus \Omega_{10R}$. The following reverse H\"oder is also recalled as a boundary version of Lemma \ref{lem:reverseHolder}, it refers to \cite[Lemma 3.4]{MPT2018} for the detailed proof, or another version in \cite[Lemma 2.5]{Phuc1}, where the integrals should be taken on arbitrary sufficiently small ball.

\begin{lemma} 
\label{lem:reverseHolderbnd} 
Let $w$ be the solution to \eqref{eq:B1}. Then, there exist two constants $\Theta = \Theta(n,p,\Lambda_1,\Lambda_2,c_0)>p$ and $C = C(n,p,\Lambda_1,\Lambda_2,c_0)>0$ such that the following estimate      
	\begin{equation}\label{eq:reverseHolderbnd}
	\left(\fint_{B_{\rho/2}(y)}|\nabla w|^{\Theta} dx\right)^{\frac{1}{\Theta}}\leq C\left(\fint_{B_{2\rho/3}(y)}|\nabla w|^p dx\right)^{\frac{1}{p}}
	\end{equation}
	holds for all  $B_{2\rho/3}(y) \subset B_{10R}(x_0)$, $y \in B_r(x_0)$. 
\end{lemma}

We next state and prove the selection Lemma \ref{lem:l2} which establishes the solution comparison gradient estimate up to the boundary, that is a version of Lemma \ref{lem:I1} up to the boundary and this preparatory lemma is very important to prove our desired results later.
\begin{lemma}
\label{lem:l2}
Let $w$ be the unique solution to equation~\eqref{eq:B1}. Then, there exists a positive constant $C = C(n,p,\Lambda_1,\Lambda_2)>0$ such that the following comparison estimate
\begin{multline}
\label{eq:lem2}
\fint_{B_{10R}(x_0)} |\nabla u - \nabla w|^p dx  \le   C  \fint_{B_{10R}(x_0)} {(|F|^p + |\nabla \sigma|^p) dx} \\ + C \left(\fint_{B_{10R}(x_0)} |\nabla u|^pdx \right)^{\frac{p-1}{p}} \left(\fint_{B_{10R}(x_0)} |\nabla \sigma|^pdx \right)^{\frac{1}{p}},
\end{multline}
holds for all $p>1$.
\end{lemma}
\begin{proof}
Similar to the proof of interior lemma \ref{lem:I1}, we firstly choose $u - w - \overline{\sigma}$ as a test function of equations~\eqref{eq:diveq} and~\eqref{eq:B1}, where $\overline{\sigma} = \sigma$ in $\overline{B}_{10R}(x_0)$, which yields that
\begin{align}
\label{eq:L1I1b}
\nonumber
\int_{B_{10R}(x_0)} \left(A(x,\nabla u) - A(x, \nabla w)\right) \nabla (u - w) dx  = \int_{B_{10R}(x_0)} \left(A(x,\nabla u) - A(x, \nabla w)\right) \nabla \sigma dx \\ + \int_{B_{10R}(x_0)} |F|^{p-2} F \nabla (u - w) dx  - \int_{B_{10R}(x_0)} |F|^{p-2} F \nabla \sigma dx.
\end{align}
The previous assumptions on the operator $A$ (see \eqref{eq:A1} and~\eqref{eq:A2} in Section \ref{sec:A}) immediately yield that there exists a positive constant $C$ depending on $\Lambda_1, \Lambda_2$ such that
\begin{multline}
\label{eq:L1I2b}
\int_{B_{10R}(x_0)}  |\nabla u - \nabla w|^p dx \le C \left(\int_{B_{10R}(x_0)} \left(|\nabla u| + |\nabla w |\right)^{p-1} |\nabla \sigma| dx\right.  \\ + \left.  \int_{B_{10R}(x_0)} |F|^{p-1} |\nabla u - \nabla w| dx + \int_{B_{10R}(x_0)} |F|^{p-1}  |\nabla \sigma| dx\right).
\end{multline}
Inequality \eqref{eq:remin} is now applied again to get
\begin{equation}
\label{eq:L10b}
\int_{B_{10R}(x_0)}  |\nabla u - \nabla w|^p dx \le C \left(I_1 + I_2 + I_3 + I_4\right),
\end{equation}
where 
\begin{align*}
& I_1 = \int_{B_{10R}(x_0)} |\nabla u|^{p-1} |\nabla \sigma| dx,\quad I_2= \int_{B_{10R}(x_0)} |\nabla u - \nabla w|^{p-1} |\nabla \sigma| dx,\\
& I_3 = \int_{B_{10R}(x_0)} |F|^{p-1} |\nabla u - \nabla w| dx, \quad \text{and}\quad I_4 = \int_{B_{10R}(x_0)} |F|^{p-1}  |\nabla \sigma| dx.
\end{align*}
For any $\varepsilon>0$, thanks to H\"older's inequality and Young's inequality, it is clearly to obtain the integral estimate for each term $I_i$ ($i=1,2,3,4$) in much the same way as \eqref{eq:estI0}, \eqref{eq:estI1}, \eqref{eq:estI2} and \eqref{eq:estI3} in previous proof of Lemma \ref{lem:I1} but on the ball $B_{10R}(x_0)$. Then, when choosing $\varepsilon = \frac{1}{2}$ small enough, the assertion of lemma is concluded.
\end{proof}

\section{Proofs of main Theorems}
\label{sec:proofs}
This section is devoted to separable proofs of our main results in Theorem \ref{theo:maintheo_lambda}, \ref{theo:main2} and some gradient norm estimates presented in Theorem \ref{theo:regularityM0} and Theorem \ref{theo:regularityMalpha}. Key ingredients of proofs include some properties of maximal/cut-off maximal functions, and the following Lemma \ref{lem:mainlem}. It can be viewed as a substitution for the Calder\'on-Zygmund-Krylov-Safonov decomposition. The reader is referred to \cite{CC1995} for the proof of this lemma.

\begin{lemma}
\label{lem:mainlem}
Let $0<\varepsilon<1$ and $R\ge R_1>0$ and the ball $Q:=B_R(x_0)$ for some $x_0\in \mathbb{R}^n$.  Let $V\subset W\subset Q$ be two measurable sets satisfying two following properties: 
\begin{itemize}
\item[(i)] $\mathcal{L}^n\left(V\right)<\varepsilon \mathcal{L}^n\left(B_{R_1}\right)$;
\item[(ii)] For all $x \in Q$ and $r \in (0,R_1]$, we have $B_r(x) \cap Q \subset W$ provided $\mathcal{L}^n\left(V \cap B_r(x)\right)\geq \varepsilon \mathcal{L}^n\left(B_r(x)\right)$.	
\end{itemize}
Then $\mathcal{L}^n\left(V\right)\leq C \varepsilon \mathcal{L}^n\left(W\right)$ for some constant $C=C(n)$.
\end{lemma}
\begin{proof}[Proof of Theorem \ref{theo:maintheo_lambda}]
Take $0<\varepsilon<1$ and $\lambda>0$. First of all, let us set:
\begin{align*}
V_\lambda &= \{{\mathbf{M}}(|\nabla u|^p)>\varepsilon^{-a}\lambda, {\mathbf{M}}(|F|^p+|\nabla \sigma|^p) \le \varepsilon^b\lambda \}\cap \Omega; \\
W_\lambda&= \{ {\mathbf{M}}(|\nabla u|^p)> \lambda\}\cap \Omega,
\end{align*}
All we need is to verify that there exists a constant $C$ such that $\mathcal{L}^n\left(V_\lambda\right) <C \varepsilon \mathcal{L}^n\left(W_\lambda\right)$. This process can be splitted into 2 steps. Let us start with the step of verification $(i)$ in Lemma \ref{lem:mainlem}.\\

\emph{Step 1.} If $V_\lambda \neq \emptyset$, then there exists $x_1 \in \Omega$ such that:
\begin{align}
\label{eq:x1_1}
{\mathbf{M}}(|\nabla u|^p)(x_1) &>\varepsilon^{-a}\lambda,
\end{align}
and
\begin{align}
\label{eq:x1_2}
{\mathbf{M}}(|F|^p+|\nabla \sigma|^p)(x_1) &\le \varepsilon^b\lambda.
\end{align}
From \eqref{eq:x1_2} and the definition of maximal function ${\mathbf{M}}$, it gives us
\begin{align*}
\sup_{\rho>0}{\fint_{B_\rho(x_1)}{(|F|^p+|\nabla \sigma|^p) dx}} \le \varepsilon^b\lambda
\end{align*}
which implies
\begin{align*}
\fint_{B_\rho(x_1)}{(|F|^p+|\nabla \sigma|^p) dx} \le \varepsilon^b \lambda, \quad \forall \rho>0,
\end{align*}
and we obtain
\begin{align}\label{eq:res3}
\int_{B_{\rho}(x_1)}{(|F|^p+|\nabla \sigma|^p) dx} \le \varepsilon^b\lambda\mathcal{L}^n\left(B_\rho(x_1)\right), \quad \forall\rho>0.
\end{align}
To prove the claim, we choose $\rho=T_0 := diam(\Omega)$ to get:
\begin{align*}
\begin{split}
\int_{\Omega}{(|F|^p+|\nabla \sigma|^p) dx} &\le \varepsilon^b\lambda \mathcal{L}^n\left(B_{T_0}(x_1)\right)\\
&\le C\varepsilon^b\lambda \left(\frac{T_0}{R_1} \right)^n \mathcal{L}^n\left(B_{R_1}(0)\right)\\
&\le C\varepsilon^b\lambda\mathcal{L}^n\left(B_{R_1}(0)\right).
\end{split}
\end{align*}
On the other hand, from \eqref{eq:x1_1} and due to the fact that ${\mathbf{M}}$ is bounded from $L^1(\mathbb{R}^n)$ into $L^{1,\infty}(\mathbb{R}^n)$, it is clearly to see that
\begin{align}
\label{eq:res4}
\mathcal{L}^n\left(V_\lambda\right) \le \mathcal{L}^n\left(\left\{ {\mathbf{M}}(|\nabla u|^p)>\varepsilon^{-a}\lambda \right\} \right) \le \frac{1}{\varepsilon^{-a} \lambda}\int_{\Omega}{|\nabla u|^p dx}.
\end{align}
It follows from Proposition \ref{prop1} that for $p>1$, we have:
\begin{align*}
\mathcal{L}^n\left(V_\lambda\right) \le \frac{1}{\varepsilon^{-a} \lambda}\int_\Omega{(|F|^p+|\nabla \sigma|^p) dx} \le {C\varepsilon^{a+b}}\mathcal{L}^n\left(B_{R_1}(0)\right)<C\varepsilon\mathcal{L}^n\left(B_{R_1}(0)\right),
\end{align*}
where the last estimate comes from the fact that $a+b>1$. 
Here, one notices that the constant $C$ depends on $n, T_0$. \bigskip

\emph{Step 2.}
Let $x_0$ be fixed in the interior of $\Omega$. We need to prove that for all $x \in Q = B_R(x_0)$ and $r \in (0,R_1]$, we have $B_r(x) \cap Q \subset W_\lambda$, if $\mathcal{L}^n\left(V_\lambda \cap B_r(x)\right) \ge \varepsilon \mathcal{L}^n\left(B_r(x)\right)$. Indeed, let us suppose that $V_\lambda \cap B_r(x) \neq \emptyset$ and $B_r(x)\cap \Omega \cap W_\lambda^c \neq \emptyset$. Then, there exist $x_2, x_3 \in B_r(x) \cap \Omega$ such that:
\begin{align}
\label{eq:x2}
{\mathbf{M}}(|\nabla u|^p)(x_2) \le \lambda,
\end{align}
and 
\begin{align}
\label{eq:x3}
{\mathbf{M}}(|F|^p+|\nabla \sigma|^p)(x_3) \le \varepsilon^b\lambda.
\end{align}
At the moment, we need to prove that there exists a constant $C = C(n,p,\Lambda_1,\Lambda_2,c_0)>0$ such that
\begin{align}
\label{eq:iigoal}
\mathcal{L}^n\left(V_\lambda \cap B_r(x)\right) < C\varepsilon\mathcal{L}^n\left(B_r(x)\right).
\end{align}
For $\rho>0$ and $y \in B_r(x)$ we have:
\begin{align}
\label{eq:res9}
\fint_{B_\rho(y)}{|\nabla u|^p dx} \le \sup\left\{\left(\sup_{\rho'<r}{\fint_{B_{\rho'}(y)}{\chi_{B_{2r}(x)}|\nabla u|^p dx}} \right); \left( \sup_{\rho' \ge r}{\fint_{B_{\rho'}(y)}{\chi_{B_{2r}(x)}|\nabla u|^p dx}} \right)\right\},
\end{align}
where $\rho' \ge r$. Since $B_{\rho'}(y) \subset B_{\rho'+r}(x) \subset B_{\rho'+2r}(x_1) \subset B_{2\rho'}(x_1)$, then:
\begin{align*}
\sup_{\rho' \ge r}{\fint_{B_{\rho'}(y)}{\chi_{B_{2r}(x)}|\nabla u|^p dx}} \le 3^n\sup_{\rho'>0}{\fint_{B_{2\rho'}(x_1)}{|\nabla u|^p dx}}.
\end{align*}
From \eqref{eq:res9} and in the use of \eqref{eq:x2}, we get that:
\begin{align}
\label{eq:res10}
\begin{split}
\fint_{B_\rho(y)}{|\nabla u|^p dx} &\le \sup \left\{{\mathbf{M}}(\chi_{B_{2r}(x)}|\nabla u|^p)(y); 3^n\sup_{\rho'>0}{\fint_{B_{2\rho'}(x_1)}{|\nabla u|^p dx}} \right\}\\
&\le \sup\left\{{\mathbf{M}}(\chi_{B_{2r}(x)}|\nabla u|^p)(y);3^n\lambda \right\}.
\end{split}
\end{align}
Taking the supremum both sides for all $\rho'>0$ it gives
\begin{align*}
{\mathbf{M}}(|\nabla u|^p)(y) \le \max\left\{ {\mathbf{M}}(\chi_{B_{2r}(x)}|\nabla u|^p)(y);3^n\lambda \right\}, \forall y \in B_r(x).
\end{align*}
Let $\varepsilon_0 = \displaystyle{\left( \frac{1}{3}\right)^{\frac{n+1}{a}}} \in (0,1)$, then for all $\lambda>0$ and $\varepsilon \in (0,\varepsilon_0)$:
\begin{align}
\label{eq:res11}
V_\lambda \cap B_r(x) = \left\{{\mathbf{M}}(\chi_{B_{2r}(x)}|\nabla u|^p)> \varepsilon^{-a}\lambda; {\mathbf{M}}(|F|^p+|\nabla \sigma|^p) \le \varepsilon^b\lambda \right\} \cap B_r(x) \cap \Omega.
\end{align}
In order to prove \eqref{eq:iigoal}, we have to consider two cases: $B_{2r}(x) \subset\subset\Omega$ (in the interior domain) and $B_{2r}(x) \cap \partial\Omega \neq \emptyset$ (on the boundary).\\ \medskip\\
	   \textbf{Case 1: $B_{2r}(x) \subset\subset\Omega$}.\\
	   For $0<R\le R_1$ and $x_0 \in \Omega$, set $B_{2R} = B_{2R}(x_0)$ and let us consider $w$ a unique solution of equation:
	   \begin{equation}
	   \label{eq:wsol_inter}
	   \begin{cases}
	   \div A(x,\nabla w) &=0, \quad \qquad  \text{in} \ B_{2r}(x),\\
	   w &= u - \sigma, \quad \text{on} \ \partial B_{2r}(x).
	   \end{cases}
	   \end{equation}
	   We have firstly that the Lebesgue measure of the set $V_\lambda \cap B_r(x)$ can be separated into 2 parts
	   \begin{align}
	   \label{eq:res15}
	   \begin{split}
	   \mathcal{L}^n\left(V_\lambda \cap B_r(x)\right) &\le \mathcal{L}^n\left(\left\{{\mathbf{M}}(\chi_{B_{2r}(x)}|\nabla u - \nabla w|^p)>\varepsilon^{-a}\lambda\right\}\cap B_r(x) \right) \\
	   &+ \mathcal{L}^n\left(\left\{{\mathbf{M}}(\chi_{B_{2r}(x)}|\nabla w|^p)>\varepsilon^{-a}\lambda \right\} \cap B_r(x) \right).
	   \end{split}
	   \end{align}
	   Each term on the right hand side could be estimated following Lemma \ref{lem:I1} as:
	   \begin{align}\label{eq:res16}
	   \begin{split}
	   & \mathcal{L}^n\left(\left\{ {\mathbf{M}}(\chi_{B_{2r}(x)}|\nabla u - \nabla w|^p)>\varepsilon^{-a}\lambda\right\}\cap B_r(x)  \right)  \\
	   & \le \frac{C}{\varepsilon^{-a} \lambda}\int_{B_{2r}(x)}{|\nabla u - \nabla w|^p dy} \\~~~~~ 
	   & \le \frac{Cr^n}{\varepsilon^{-a} \lambda} \left[\fint_{B_{2r}(x)}{(|F|^p+|\nabla\sigma|^p)dy}+  \left(\fint_{B_{2r}(x)} |\nabla u|^pdy \right)^{\frac{p-1}{p}} \left(\fint_{B_{2r}(x)} |\nabla \sigma|^pdy \right)^{\frac{1}{p}} \right]
	   \end{split}
	   \end{align}
	   and 
	   \begin{align}\label{eq:res17}
	   \mathcal{L}^n\left(\left\{{\mathbf{M}}(\chi_{B_{2r}(x)}|\nabla w|^p)>\varepsilon^{-a}\lambda \right\} \cap B_r(x) \right) \le \frac{Cr^n}{(\varepsilon^{-a}\lambda)^{\frac{\Theta}{p}}}\fint_{B_{2r}(x)}{|\nabla w|^{\Theta} dy},
	   \end{align}
where $\Theta = \Theta(n,p,\Lambda_1,\Lambda_2,c_0)>p$ and the constant $C>0$ depending on $n,p,\Lambda_1,\Lambda_2,c_0$, appearing in the use of reverse H\"older's inequality \eqref{eq:reverseHolder} as:
   \begin{equation*}
 \left(  \fint_{B_{2r}(x)}{|\nabla w|^{\Theta} dy}\right)^{1/\Theta}\leq C  \left(  \fint_{B_{4r}(x)}{|\nabla w|^{p} dy}\right)^{1/p}.
   \end{equation*}
The parameter $\displaystyle{a=\frac{p}{\Theta}}$ is taken into account and it follows again the comparison estimate found in Lemma \ref{lem:I1} to obtain
		\begin{align}
		\label{eq:res18}
		\begin{split}
		\fint_{B_{4r}(x)}{|\nabla w|^p dy} &\le C\fint_{B_{4r}(x)}{|\nabla u|^p dy} + C\fint_{B_{4r}(x)}{|\nabla u - \nabla w|^p dy}\\
		&\le C\fint_{B_{4r}(x)}{|\nabla u|^p dy} +C\left[\fint_{B_{4r}(x)}{(|F|^p+|\nabla\sigma|^p)dy} \right.\\
		&~~~~~~~+\left. \left(\fint_{B_{4r}(x)} |\nabla u|^pdy \right)^{\frac{p-1}{p}} \left(\fint_{B_{4r}(x)} |\nabla \sigma|^pdy \right)^{\frac{1}{p}} \right].
		\end{split}
		\end{align}		
	   On the other hand, as $|x-x_2|<r$ yields that$B_{4r}(x) \subset B_{4r}(x_2)$, then by the fact in \eqref{eq:x2} it gets
	   \begin{align}
	   \label{eq:res19}
	   \begin{split}
	   \fint_{B_{4r}(x)}{|\nabla u|^p dy} &\le C \fint_{B_{4r}(x_2)}{|\nabla u|^p dy} \le C\sup_{\rho>0}{\fint_{B_{4r}(x_2)}{|\nabla u|^p dy}}\\
	    &= C{\mathbf{M}}(|\nabla u|^p)(x_2) \le C\lambda.
	    \end{split}
	   \end{align}
	   Analogously, as $|x-x_3|<r$, $B_{4r}(x) \subset B_{4r}(x_3)$, then for all $\rho>0$ and from \eqref{eq:x3} we have
	   \begin{align}
	   \label{eq:res20}
	   \begin{split}
	   \fint_{B_{4r}(x)}{(|F|^p+|\nabla \sigma|^p)dy} &\le C\fint_{B_{4r}(x_3)}{(|F|^p+|\nabla \sigma|^p)dy} \\
	   & \le C\sup_{\rho>0}{\fint_{B_\rho(x_3)}{(|F|^p+|\nabla \sigma|^p)dy}} \\
	    &= C{\mathbf{M}}(|F|^p+|\nabla \sigma|^p)(x_3) \le C\varepsilon^b\lambda.
	    \end{split}
	   \end{align}
		Applying to \eqref{eq:res16} we obtain that
		\begin{align*}
		\begin{split}
		\mathcal{L}^n\left(\left\{ {\mathbf{M}}(\chi_{B_{2r}(x)}|\nabla u - \nabla w|^p)>\varepsilon^{-a}\lambda\right\}\cap B_r(x)  \right) &\le \frac{Cr^n}{\varepsilon^{-a} \lambda}\left(\varepsilon^b\lambda+\varepsilon^{\frac{b}{p}}\lambda \right)\\
		&= {Cr^n}\varepsilon\left(\varepsilon^{a+b-1}+1 \right);
		\end{split}
\end{align*}		   
	   and to \eqref{eq:res17}:
	   \begin{align*}
	   \begin{split}
	   \mathcal{L}^n\left(\left\{{\mathbf{M}}(\chi_{B_{2r}(x)}|\nabla w|^p)>\varepsilon^{-a}\lambda \right\} \cap B_r(x) \right) \le \frac{Cr^n}{\varepsilon^{-1}\lambda^{\frac{\Theta}{p}}} \left[\lambda + \varepsilon^b\lambda+\varepsilon^{\frac{b}{p}}\lambda \right]^{\frac{\Theta}{p}}\\
	    = {Cr^n}{\varepsilon}\left(1+\varepsilon^b+\varepsilon^{\frac{b}{p}} \right)^{\frac{\Theta}{p}}.
	   \end{split}
	   \end{align*}
	   Combining these above estimates into \eqref{eq:res15}
	   \begin{align}
	   \mathcal{L}^n\left(V_\lambda \cap B_r(x)\right) \le {Cr^n}{\varepsilon}\left[1+\varepsilon^{a+b-1}+ \left(1+\varepsilon^b+\varepsilon^{\frac{b}{p}} \right)^{\frac{\Theta}{p}} \right]<C\varepsilon r^n,
	   \end{align}
	   which establishes the desired result.\bigskip
	   
	    \textbf{Case 2: $B_{2r}(x) \cap \partial\Omega \neq \emptyset$}. Let $x_4 \in \partial \Omega$ such that $|x_4-x| = dist(x,\partial\Omega) \le 2r$, then $B_{2r}(x) \subset B_{10r}(x_4)$. Application of Lemma \ref{lem:l2} on the ball $B_{10r}(x_4)$ enables us to get the existence of a constant $C = C(n,p,\Lambda_1,\Lambda_2,c_0)>0$ such that: 
	    \begin{multline*}
	    \fint_{B_{10r}(x_4)} |\nabla u - \nabla w|^p dy  \le   C  \fint_{B_{10r}(x_4)} {(|F|^p + |\nabla \sigma|^p) dy} \\ + C \left(\fint_{B_{10r}(x_4)} |\nabla u|^pdy \right)^{\frac{p-1}{p}} \left(\fint_{B_{10r}(x_4)} |\nabla \sigma|^pdy \right)^{\frac{1}{p}}.
	    \end{multline*}
		As a boundary version of \eqref{eq:res15}, on the ball $B_{10r}(x_4)$ one has
		\begin{align}
		\label{eq:res1}
		\begin{split}
	   \mathcal{L}^n\left(V_\lambda \cap B_r(x)\right) &\le \mathcal{L}^n\left(\left\{{\mathbf{M}}(\chi_{B_{10r}(x_4)}|\nabla u - \nabla w|^p)>\varepsilon^{-a}\lambda\right\}\cap B_r(x) \right) \\
	   &+ \mathcal{L}^n\left(\left\{{\mathbf{M}}(\chi_{B_{10r}(x_4)}|\nabla w|^p)>\varepsilon^{-a}\lambda \right\} \cap B_r(x) \right).
	   \end{split}
		\end{align}
	   From Lemma \ref{lem:l2} we obtain each term on the right hand side of \eqref{eq:res1} as
	   \begin{align*}
	   \begin{split}
	   &\mathcal{L}^n\left(\left\{ {\mathbf{M}}(\chi_{B_{10r}(x_4)}|\nabla u - \nabla w|^p)>\varepsilon^{-a}\lambda\right\}\cap B_r(x)  \right) \le \frac{Cr^n}{\varepsilon^{-a} \lambda}\fint_{B_{10r}(x_4)}{|\nabla u - \nabla w|^p dy} \\~~~~~ &\le \frac{Cr^n}{\varepsilon^{-a} \lambda} \left[\fint_{B_{10r}(x_4)}{(|F|^p+|\nabla\sigma|^p)dy}+ \left(\fint_{B_{10r}(x_4)} |\nabla u|^pdy \right)^{\frac{p-1}{p}} \left(\fint_{B_{10r}(x_4)} |\nabla \sigma|^pdy \right)^{\frac{1}{p}} \right];
	   \end{split}
	   \end{align*}
	   and according to the reverse H\"older inequality in Lemma \ref{lem:reverseHolderbnd}, there exist $\Theta = \Theta(n,p,\Lambda_1,\Lambda_2,c_0)>p$ and a constant $C = C(n,p,\Lambda_1,\Lambda_2,T_0,\Theta)$ such that
	   \begin{align*}
	   \begin{split}
	  & \mathcal{L}^n\left(\left\{{\mathbf{M}}(\chi_{B_{10r}(x_4)}|\nabla w|^p)>\varepsilon^{-a}\lambda \right\} \cap B_r(x) \right) \le \frac{Cr^n}{\left(\varepsilon^{-a}\lambda\right)^{\frac{\Theta}{p}}}\fint_{B_{10r}(x_4)}{|\nabla w|^\Theta dy}\\
	  &~~~~~~~\le \frac{Cr^n}{\varepsilon^{-1}\lambda^{\frac{\Theta}{p}}} \left[ \fint_{B_{14r}(x_4)}{|\nabla u|^p dy} + \fint_{B_{14r}(x_4)}{(|F|^p+|\nabla\sigma|^p)dy} \right.\\
		&~~~~~~~~~~~~+\left. \left(\fint_{B_{14r}(x_4)} |\nabla u|^pdy \right)^{\frac{p-1}{p}} \left(\fint_{B_{14r}(x_4)} |\nabla \sigma|^pdy \right)^{\frac{1}{p}} \right]^{\frac{\Theta}{p}}.
	   \end{split}
	   \end{align*}
	   For $x_2, x_3$ determined in the previous case and the definition of $x_4$, since $dist(x,\partial\Omega) \le 2r$, we can check easily that
	   \begin{align*}
	   \begin{split}
	   \overline{B_{14r}(x_4)} \subset \overline{B_{16r}(x)} \subset \overline{B_{17r}(x_2)},\\
	   \overline{B_{14r}(x_4)} \subset \overline{B_{16r}(x)} \subset \overline{B_{17r}(x_3)},
	   \end{split}
	   \end{align*}
and it gives us two following inequalities
		\begin{align*}
		\int_{B_{14r}(x_4)}{|\nabla u|^p dy} \le C{\mathbf{M}}(|\nabla u|^p)(x_2) \le C\lambda,
		\end{align*}
and
		\begin{align*}
		\int_{B_{14r}(x_4)}{(|F|^p+|\nabla \sigma|^p)dy} \le C{\mathbf{M}}(|F|^p+|\nabla \sigma|^p)(x_3) \le C\varepsilon^b\lambda.
		\end{align*}
		Therefore, we also obtain the fact that $\mathcal{L}^n\left(V_\lambda \cap B_r(x)\right) \le C\varepsilon r^n$ and the proof is completed when we apply exactly Lemma \ref{lem:mainlem} by contradiction. That means, there exists a constant $C$ depending only on $n,p,\Lambda_1,\Lambda_2,T_0,c_0,r_0,\Theta$ such that $\mathcal{L}^n\left(V_\lambda\right) \le C\varepsilon\mathcal{L}^n\left(W_\lambda\right)$ holds for any $\lambda>0$ and all $\varepsilon \in (0,\varepsilon_0)$.
\end{proof} 

Theorem \ref{theo:maintheo_lambda} gives us idea to get the gradient norm estimate of solutions to \eqref{eq:diveq} in Lorentz space. Having disposed of this preliminary step, we can now proceed to the proof of Theorem \ref{theo:regularityM0}.

\begin{proof}[Proof of Theorem \ref{theo:regularityM0}]
The definition of norm in Lorentz space $L^{q,s}(\Omega)$ in \eqref{eq:lorentz} gives:
\begin{align*}
\|{\mathbf{M}}(|\nabla u|^p)\|^s_{L^{s,q}(\Omega)} = q \int_0^\infty{\lambda^s \mathcal{L}^n\left(\{{\mathbf{M}}(|\nabla u|^p)>\lambda\} \right)^{\frac{s}{q}}\frac{d\lambda}{\lambda}}.
\end{align*}
By changing the variable $\lambda$ to $\varepsilon^{-a}\lambda$ within the integral above, we get that:
\begin{align*}
\|{\mathbf{M}}(|\nabla u|^p)\|^s_{L^{s,q}(\Omega)} = \varepsilon^{-as}q\int_0^\infty{\lambda^s\mathcal{L}^n\left(\{{\mathbf{M}}(|\nabla u|^p)>\varepsilon^{-a}\lambda\} \right)^{\frac{s}{q}}\frac{d\lambda}{\lambda}},
\end{align*}
and Theorem \ref{theo:maintheo_lambda} makes it that
\begin{align*}
\mathcal{L}^n\left(\{{\mathbf{M}}(|\nabla u|^p)>\varepsilon^{-a}\lambda\}\right) &\le C\varepsilon \mathcal{L}^n\left(\{{\mathbf{M}}(|\nabla u|^p)>\lambda\}\cap\Omega \right)\\ &~~+ \mathcal{L}^n\left(\{{\mathbf{M}}(|F|^p+|\nabla u|^p)>\varepsilon^b\lambda\}\cap\Omega \right),
\end{align*}
one obtains:
\begin{align*}
\|{\mathbf{M}}(|\nabla u|^p)\|^s_{L^{s,q}(\Omega)} &\le C\varepsilon^{-as+\frac{s}{q}}q\int_0^\infty{\lambda^s\mathcal{L}^n\left(\{{\mathbf{M}}(|\nabla u|^p)>\lambda\}\cap\Omega \right)^{\frac{s}{q}}\frac{d\lambda}{\lambda}}\\
&~~~+ C\varepsilon^{-as}q\int_0^\infty{\lambda^s\mathcal{L}^n\left(\{{\mathbf{M}}(|F|^p+|\nabla \sigma|^p)>\varepsilon^b\lambda\}\cap\Omega \right)^{\frac{s}{q}}\frac{d\lambda}{\lambda}}.
\end{align*}
Performing change of variables in the second integral on right-hand side, yields that
\begin{align*}
\|{\mathbf{M}}(|\nabla u|^p)\|^s_{L^{s,q}(\Omega)} &\le C\varepsilon^{-as+\frac{s}{q}}\|{\mathbf{M}}\left(|\nabla u|^p \right)\|^s_{L^{s,q}(\Omega)}\\ &~~+C\varepsilon^{-as-bs}\|{\mathbf{M}}(|F|^p+|\nabla \sigma|^p)\|^s_{L^{s,q}(\Omega)}.
\end{align*}
Therefore, for $0<s<\infty$ and $0<q<\frac{\Theta}{p}$ with $s\left(\frac{1}{q}-a \right)>0$ we choose $\varepsilon_0>0$ sufficiently small such that:
\begin{align*}
C\varepsilon_0^{s\left(\frac{1}{q}-a\right)} \le \frac{1}{2},
\end{align*}
and our gradient norm then holds for all $\varepsilon \in (0,\varepsilon_0)$. This is precisely the assertion of Theorem \ref{theo:regularityM0}.
\end{proof}

Our next objective is to prove the stronger result than in Theorem \ref{theo:maintheo_lambda}, that exploits the cut-off fractional maximal functions effectively in our study. We are then led to the following proof of Theorem \ref{theo:main2}.
\bigskip

\begin{proof}[Proof of Theorem \ref{theo:main2}]
Let $\lambda>0$ and $u$ be a solution to~\eqref{eq:diveq}. The proof is now proceeded analogously by applying Lemma~\ref{lem:mainlem} and preparatory lemmas in Section \ref{sec:cutoff}. We are left with the task of verifying that there exists a constant $C>0$ such that $\mathcal{L}^n\left(V_{\lambda}^{\alpha}\right) \le C \varepsilon \mathcal{L}^n\left(W_{\lambda}^{\alpha}\right)$, for $\varepsilon>0$ small enough. The process is divided into 2 steps.

{\bf Step 1:} Without loss of generality, we can assume that $V_{\lambda}^{\alpha} \neq \emptyset$ (indeed, if this set is empty, the proof is then straightforward). It follows that there exists at least a point $x_1 \in \Omega$ such that
\begin{align}
\label{eq:Mx1}
{\mathbf{M}} {\mathbf{M}}_{\alpha}(|\nabla u|^p)(x_1) >\varepsilon^{-a} \lambda, \ \mbox{ and } \ {\mathbf{M}}_{\alpha}(|F|^p + |\nabla\sigma|^p)(x_1) \le \varepsilon^{b} \lambda.
\end{align}
From the boundedness of maximal function ${\mathbf{M}}$ from $L^1(\mathbb{R}^n)$ into $L^{1,\infty}(\mathbb{R}^n)$, we have
\begin{align}
\label{eq:V1}
\mathcal{L}^n\left(V_{\lambda}^{\alpha}\right)  \le \mathcal{L}^n\left(\left\{{\mathbf{M}} {\mathbf{M}}_{\alpha}(|\nabla u|^p) >\varepsilon^{-a} \lambda \right\}\right)    \le \frac{C}{\varepsilon^{-a} \lambda} \int_{\Omega} {\mathbf{M}}_{\alpha}(|\nabla u|^p)(y) dy.
\end{align}
On the other hand, we also have
\begin{align*}
 \int_{\Omega} \mathbf{M}_\alpha(|\nabla u|^p)(y) dy&= \int_{0}^{\infty} \mathcal{L}^n\left(\left\{y\in \Omega: \ \mathbf{M}_\alpha(|\nabla u|^p)(y)>\lambda\right\}\right) d\lambda\\& \leq  CT_0^n\lambda_0+\int_{\lambda_0}^{\infty} \mathcal{L}^n\left(\left\{y \in \Omega: \ \mathbf{M}_\alpha(|\nabla u|^p)(y)>\lambda\right\}\right) d\lambda\\&\leq C T_0^n\lambda_0+ C \left(\int_{\Omega}|\nabla u|^p(y)dy\right)^{\frac{n}{n-\alpha}}\int_{\lambda_0}^{\infty}\lambda^{-\frac{n}{n-\alpha}}d\lambda\\&= CT_0^n\lambda_0+ C \left(\int_{\Omega}|\nabla u|^p(y)dy\right)^{\frac{n}{n-\alpha}}\lambda_0^{-\frac{\alpha}{n-\alpha}}.
\end{align*}
Choosing
\begin{equation*}
\lambda_0=T_0^{-n+\alpha}\int_{\Omega}|\nabla u|^p(y)dy,
\end{equation*}
yields that
\begin{align}\label{eq:Mup}
\int_{\Omega} \mathbf{M}_\alpha(|\nabla u|^p)(y) dy&\leq CT_0^{\alpha}\int_{\Omega}|\nabla u|^p(y)dy.
\end{align}
Let us apply Proposition~\ref{prop1}, \eqref{eq:Mx1} and \eqref{eq:Mup} simultaneously, it gets
\begin{align}\nonumber
\int_{\Omega} {\mathbf{M}}_{\alpha}(|\nabla u|^p)(y) dy & \le CT_0^{\alpha}\int_{\Omega}(|F|^p + |\nabla \sigma|^p)(y)dy   \\ \nonumber
&\le C T_0^n T_0^{\alpha} \fint_{B_{T_0}(x_1)} (|F|^p + |\nabla \sigma|^p)(y)dy \\ \nonumber
& \le  CT_0^n {\mathbf{M}}_{\alpha}(|F|^p + |\nabla \sigma|^p)(y)dy \\ \nonumber
& \le C\varepsilon^{b} \lambda \mathcal{L}^n\left(B_{T_0}(x_1)\right) \\ \nonumber
&\le C\left(\frac{T_0}{R}\right)^n \varepsilon^{b} \lambda \mathcal{L}^n\left(B_R(0)\right)\\ \label{eq:M}
& \le C \varepsilon^{b} \lambda \mathcal{L}^n\left(B_R(0)\right).
\end{align}
Thanks to the argument~\eqref{eq:M} to \eqref{eq:V1}, we conclude that
$$ \mathcal{L}^n\left(V_{\lambda}^{\alpha}\right)  \le C \varepsilon^{a+b} \mathcal{L}^n\left(B_R(0)\right) \le C \varepsilon \mathcal{L}^n\left(B_R(0)\right),$$
where constant $C$ depends on $n$, $T_0$ and with $a = \frac{p}{\Theta}$, the parameter $b$ is chosen later such that $a+b\ge 1$. Step 1 is thereby completed.\\

{\bf Step 2:} Let $x_0 \in \Omega$, we verify in step 2 that for all $x \in B_{T_0}(x_0)$, $r \in (0,2R]$ and $\lambda>0$, it may be concluded that
$$\mathcal{L}^n\left(V_{\lambda}^{\alpha} \cap B_r(x)\right) \ge C \varepsilon \mathcal{L}^n\left(B_r(x)\right) \Longrightarrow B_r(x) \cap \Omega \subset W_{\lambda}^{\alpha}.$$
According to Lemma~\ref{lem:mainlem}, the contradiction will point out that $\mathcal{L}^n\left(V_{\lambda}^{\alpha} \cap B_r(x)\right) < C\varepsilon \mathcal{L}^n\left(B_r(x)\right)$. Let us firstly assume that $V_{\lambda}^{\alpha} \cap B_r(x) \neq \emptyset$ and $B_r(x) \cap \Omega \cap (W_{\lambda}^{\alpha})^c \neq \emptyset$, that means there exist $x_2, x_3 \in B_r(x) \cap \Omega$ such that
\begin{align}
\label{eq:MMx2} &{\mathbf{M}}{\mathbf{M}}_{\alpha} (|\nabla u|^p)(x_2) \le \lambda, \\
\label{eq:MFx3} &{\mathbf{M}}_{\alpha}(|F|^p + |\nabla \sigma|^p)(x_3) \le \varepsilon^b\lambda.
\end{align}
Applying Lemma~\ref{lem:Malpha} gives us the following assertion
\begin{align}
\begin{split}
\label{eq:Vlam}
\mathcal{L}^n\left(V_{\lambda}^{\alpha} \cap B_r(x)\right)&\le  \mathcal{L}^n\left(\left\{{\mathbf{M}} {\mathbf{M}}_{\alpha}(|\nabla u|^p) >\varepsilon^{-a} \lambda \right\} \cap B_r(x)\right)\\ &\le  \max \left\{Q_1, Q_2, Q_3 \right\},
\end{split}
\end{align}
where
$$ Q_1 = \mathcal{L}^n\left(\left\{{\mathbf{M}}^r  {\mathbf{M}}_{\alpha}^r(|\nabla u|^p) >\varepsilon^{-a} \lambda \right\}\cap B_r(x)\right),$$ 
$$ Q_2 = \mathcal{L}^n\left(\left\{{\mathbf{M}}^r {\mathbf{T}}_{\alpha}^r(|\nabla u|^p) > \varepsilon^{-a}\lambda\right\}\cap B_r(x)\right),$$
and
$$ Q_3 = \mathcal{L}^n\left(\left\{{\mathbf{T}}^r {\mathbf{M}}_{\alpha}(|\nabla u|^p) >\varepsilon^{-a} \lambda \right\}\cap B_r(x)\right).$$ 
For any $y \in B_r(x)$, it is easy to check that $B_{\rho}(y) \subset B_{2\rho}(x) \subset B_{3\rho}(x_2), \ \forall \rho \ge r$. From what has already been proved in Lemma~\ref{lem:Tr} (apply with $\alpha=0$ and $f = {\mathbf{M}}_{\alpha}(|\nabla u|^p)$), we then use \eqref{eq:MMx2} to obtain that
\begin{align*}
{\mathbf{T}}^r {\mathbf{M}}_{\alpha}(|\nabla u|^p)(y) \le 3^{n} {\mathbf{M}} {\mathbf{M}}_{\alpha}(|\nabla u|^p)(x_2) \le 3^{n} \lambda.
\end{align*}
Moreover, for any $\rho \in (0,r)$, $y \in B_r(x)$ and $z \in B_{\rho}(y)$, since $B_{\delta}(z) \subset B_{\delta + 3r}(x_2)$ for all $\delta \ge r$, it turns out that
\begin{align*}
\mathbf{T}^r_{\alpha}(|\nabla u|^p)(z) & = \sup_{\delta\ge r} \delta^{\alpha-n}\int_{B_{\delta}(z)} |\nabla u|^p(\xi)d\xi \\
& \le \sup_{\delta\ge r}\left(\frac{3r+\delta}{\delta}\right)^{n-\alpha}(3r+\delta)^{\alpha-n}\int_{B_{\delta+3r}(x_2)} |\nabla u|^p(\xi)d\xi \\
& \le 4^{n-\alpha} \mathbf{M}_{\alpha}(|\nabla u|^p)(x_2),
\end{align*}
yields the estimate
\begin{align*}
 {\mathbf{M}}^r {\mathbf{T}}_{\alpha}^r(|\nabla u|^p)(y)  & = \sup_{0<\rho<r} \fint_{B_{\rho}(y)} {\mathbf{T}}^r_{\alpha} (|\nabla u|^p)(z)dz \\ 
 &\le 4^{n-\alpha} \mathbf{M}_{\alpha}(|\nabla u|^p)(x_2) \le 4^{n-\alpha} \lambda.
\end{align*}
It follows that $Q_2 = Q_3 = 0$ for every $\varepsilon^{-a}>4^{n}>\max\{4^{n-\alpha},3^n\}$.
Therefore, for all $\lambda>0$, it is possible to choose $\varepsilon_0$ such that $\varepsilon_0^{-a} > 4^{n}$ and then we obtain
\begin{align*}
\mathcal{L}^n\left(V_{\lambda}^{\alpha} \cap B_r(x)\right) \le Q_1, \ \mbox{ for all } \ \varepsilon \in (0,\varepsilon_0).
\end{align*}
Analogously, we need only to consider two cases: $B_{2r}(x) \subset\subset\Omega$ (in the interior domain) and $B_{2r}(x) \cap \partial\Omega \neq \emptyset$ (on the boundary).\\ \medskip\\
\emph{Case 1.} $B_{2r}(x) \subset \subset \Omega$.\\
Again, let us consider $w$ the unique solution to the equation \eqref{eq:wsol_inter} and apply Lemma~\ref{lem:MrMr} to obtain that
\begin{align}\label{eq:interset}
\mathcal{L}^n\left(V_{\lambda}^{\alpha} \cap B_r(x)\right)  \le \mathcal{L}^n\left( \left\{{\mathbf{M}}^{2r}_{\alpha}(\chi_{B_{2r}(x)}|\nabla u|^p) > \varepsilon^{-a} \lambda\right\} \cap B_r(x) \right) \le S_1 + S_2,
\end{align}
where 
$$S_1 =  \mathcal{L}^n\left( \left\{{\mathbf{M}}^{2r}_{\alpha}(\chi_{B_{2r}(x)}|\nabla u-\nabla w|^p) > \varepsilon^{-a} \lambda\right\} \cap B_r(x) \right),$$
and
$$S_2 =  \mathcal{L}^n\left( \left\{{\mathbf{M}}^{2r}_{\alpha}(\chi_{B_{2r}(x)}|\nabla w|^p) > \varepsilon^{-a} \lambda\right\} \cap B_r(x) \right).$$
The bounded property of fractional maximal function ${\mathbf{M}}_{\alpha}$ deduces that
\begin{align*}
S_1  &\le \frac{C}{\left(\varepsilon^{-a}\lambda\right)^{\frac{n}{n-\alpha}}} \left(\int_{B_{2r}(x)} |\nabla u - \nabla w|^p dy\right)^{\frac{n}{n-\alpha}}\\
&\le \frac{C}{\left(\varepsilon^{-a}\lambda\right)^{\frac{n}{n-\alpha}}} r^{\frac{n^2}{n-\alpha}}\left(\fint_{B_{2r}(x)} |\nabla u - \nabla w|^p dy\right)^{\frac{n}{n-\alpha}}\\
&= \frac{C}{\left(\varepsilon^{-a}\lambda\right)^{\frac{n}{n-\alpha}}} r^n \left(r^{\alpha}\fint_{B_{2r}(x)} |\nabla u - \nabla w|^p dy\right)^{\frac{n}{n-\alpha}}.
\end{align*}
Thanks to Lemma \ref{lem:I1} one has
\begin{align}
\label{eq:cas1_4}
\begin{split}
r^{\alpha}\fint_{B_{2r}(x)} |\nabla u - \nabla w|^p dy &\le C r^{\alpha}\fint_{B_{2r}(x)}{(|F|^p+|\nabla \sigma|^p)dy}\\ &+C\left(r^{\alpha}\fint_{B_{2r}(x)}|\nabla u|^pdy \right)^{\frac{p-1}{p}}\left(r^{\alpha}\fint_{B_{2r}(x)}|\nabla \sigma|^pdy\right)^{\frac{1}{p}}.
\end{split}
\end{align}
As $|x-x_2|<r$ and $|x-x_3|<r$, it is not difficult to show that both $B_{2r}(x) \subset B_{4r}(x_2)$ and $B_{2r}(x) \subset B_{4r}(x_3)$ hold. Then, in the use of \eqref{eq:MMx2} and \eqref{eq:MFx3} it gets that
\begin{align*}
\displaystyle{r^{\alpha}\fint_{B_{2r}(x)}{|\nabla u|^p dy}} \le C \displaystyle{r^{\alpha}\fint_{B_{4r}(x_2)}{|\nabla u|^p dy}} \le C {\mathbf{M}\mathbf{M}}_\alpha (|\nabla u|^p)(x_2) \le C\lambda;
\end{align*}
and
\begin{align*}
\displaystyle{r^{\alpha}\fint_{B_{2r}(x)}{(|F|^p+|\nabla \sigma|^p) dy}} & \le C\displaystyle{r^{\alpha}\fint_{B_{4r}(x_3)}{(|F|^p+|\nabla \sigma|^p) dy}} \\ 
& \le C {\mathbf{M}}_{\alpha}(|F|^p+|\nabla \sigma|^p)(x_3) \\ 
&\le C\varepsilon^b\lambda.
\end{align*}
When combined them with the inequality \eqref{eq:cas1_4}, one gave:
\begin{align*}
\left(r^{\alpha}\fint_{B_{2r}(x)} |\nabla u - \nabla w|^p dy\right)^{\frac{n}{n-\alpha}} &\le C\left[\varepsilon^b\lambda+\lambda^{\frac{p-1}{p}}(\varepsilon^b\lambda)^{\frac{1}{p}} \right]^{\frac{n}{n-\alpha}}\\
&\le C\varepsilon^{\frac{bn}{p(n-\alpha)}}\left(\varepsilon^{\frac{b(p-1)}{p}} + 1 \right)^{\frac{n}{n-\alpha}}\lambda^{\frac{n}{n-\alpha}}.
\end{align*}
Therefore,
\begin{align}\label{eq:S1}
S_1 \le Cr^n \varepsilon^{\left(a+\frac{b}{p} \right)\frac{n}{n-\alpha}}\left(\varepsilon^{\frac{b(p-1)}{p}} +1\right)^{\frac{n}{n-\alpha}} \le Cr^n\varepsilon^{\left(a+\frac{b}{p} \right)\frac{n}{n-\alpha}}.
\end{align}
For the estimation of $S_2$, thanks to the reverse H\"older inequality in Lemma \ref{lem:reverseHolder}, there exist $\Theta=\Theta(n,p,\Lambda_1,\Lambda_2,c_0)>p$ and a constant $C=C(n,p,\Lambda_1,\Lambda_2,c_0,\Theta)>0$ to find that:
\begin{multline}
\label{eq:S2est}
S_2 \le \frac{Cr^n}{(\varepsilon^{-a} \lambda)^{\frac{\Theta}{p}\frac{n}{n-\alpha}}} \left(r^{\alpha} \fint_{B_{2r}(x)}{|\nabla w|^\Theta dy}\right)^{\frac{n}{n-\alpha}} \\ \le \frac{Cr^n}{(\varepsilon^{-a} \lambda)^{\frac{\Theta}{p}\frac{n}{n-\alpha}}} \left( r^{\alpha}\fint_{B_{4r}(x)}{|\nabla w|^p dy}\right)^{\frac{\Theta n}{p(n-\alpha)}}.
\end{multline}
Following Lemma \ref{lem:I1}, one has
\begin{align*}
r^{\alpha}\fint_{B_{4r}(x)}{|\nabla w|^p dy} &\le C r^{\alpha}\fint_{B_{4r}(x)}{|\nabla u|^p dy} + Cr^{\alpha} \fint_{B_{4r}(x)}{|\nabla u-\nabla w|^p dy}\\ 
& \le C r^{\alpha}\fint_{B_{4r}(x)}{|\nabla u|^p dy} + C r^{\alpha}\fint_{B_{4r}(x)} {(|F|^p+|\nabla \sigma|^p)dy} \\ & \qquad + C \left(r^{\alpha}\fint_{B_{4r}(x)}{|\nabla u|^p dy} \right)^{\frac{p-1}{p}}\left(r^{\alpha}\fint_{B_{4r}(x)}{|\nabla\sigma|^p dy} \right)^{\frac{1}{p}} \\
&\le C\left[\lambda + \varepsilon^b\lambda + \lambda^{\frac{p-1}{p}}(\varepsilon^{b}\lambda)^{\frac{1}{p}} \right]\\
&= C\lambda\left(1+\varepsilon^b+\varepsilon^{\frac{b}{p}}\right).
\end{align*}
Thus, applying this to \eqref{eq:S2est} we finally get the estimation for $S_2$.
\begin{align}\label{eq:S2}
S_2 \le Cr^n\varepsilon^{\frac{a\Theta}{p}\frac{n}{n-\alpha}}\lambda^{\frac{-\Theta n}{p(n-\alpha)}}\left[\lambda\left(1+\varepsilon^b+\varepsilon^{\frac{b}{p}}\right) \right]^{\frac{\Theta n}{p(n-\alpha)}} \le Cr^n\varepsilon^{\frac{a\Theta n}{p(n-\alpha)}}.
\end{align}
For the chosen parameter $a = \frac{p(n-\alpha)}{n\Theta}$, $b$ is then clarified such that
\begin{align*}
a+\frac{b}{p}=1.
\end{align*}
It follows that
\begin{align*}
\left(a+\frac{b}{p} \right)\frac{n}{n-\alpha}\ge 1, \quad \mbox{ and } \quad a + b > a+\frac{b}{p}=1.
\end{align*}
Therefore, from \eqref{eq:interset}, \eqref{eq:S1} and \eqref{eq:S2} it may conclude that
\begin{align*}
\mathcal{L}^n\left(V_{\lambda}^{\alpha}\cap B_r(x)\right) \le C r^n \varepsilon.
\end{align*}
\bigskip
\emph{Case 2.} $B_{2r}(x) \cap \partial\Omega \neq \emptyset$. First of all, let us take a point $x_4 \in \partial\Omega$ satisfying $|x-x_4|=d(x,\partial\Omega)<2r$. Then, it is clear to see that $B_{2r}(x) \subset B_{10r}(x_4)$.

Let $w$ be the unique solution to the equation:
\begin{align}
\label{eq:solw_bound}
\begin{cases}
	   \div A(x,\nabla w) &=0, \quad \quad  \text{in} \ B_{10r}(x_4)\\
	   w &= u - \sigma, \quad \text{on} \ \partial B_{10r}(x_4).
\end{cases}
\end{align}
On the ball $B_{10r}(x_4)$, applying Lemma \ref{lem:l2} the comparison estimate between $\nabla u$ and $\nabla w$, one obtains:
\begin{multline*}
\fint_{B_{10r}(x_4)}{|\nabla u - \nabla w|^p dy} \le C\fint_{B_{10r}(x_4)}{(|F|^p+|\nabla \sigma|^p)dy}\\ + C\left(\fint_{B_{10r}(x_4)}{|\nabla u|^p dy} \right)^{\frac{p-1}{p}}\left(\fint_{B_{10r}(x_4)}{|\nabla \sigma|^p dy} \right)^{\frac{1}{p}}.
\end{multline*}
As the boundary version of \eqref{eq:interset}, it allows us to write:
\begin{align}
\label{eq:boundset}
\begin{split}
\mathcal{L}^n\left(V_\lambda^\alpha \cap B_r(x)\right) &\le \mathcal{L}^n\left(\{{\mathbf{M}}_{\alpha}^{2r}\left(\chi_{B_{10r}(x_4)}|\nabla u|^p \right)>\varepsilon^{-a}\lambda\}\cap B_r(x) \right)\\
&\le \mathcal{L}^n\left(\{{\mathbf{M}}_{\alpha}^{2r}\left(\chi_{B_{10r}(x_4)}|\nabla u-\nabla w|^p \right)>\varepsilon^{-a}\lambda\}\cap B_r(x) \right)\\
&~~~~+ \mathcal{L}^n\left(\{{\mathbf{M}}_{\alpha}^{2r}\left(\chi_{B_{10r}(x_4)}|\nabla w|^p \right)>\varepsilon^{-a}\lambda\}\cap B_r(x) \right)\\
&= S_1^B + S_2^B.
\end{split}
\end{align}
Note that each term on the right-hand side of \eqref{eq:boundset} can be estimated similarly to the previous case. More precisely, one respects to:
\begin{align}\label{eq:S1B}
\begin{split}
S_1^B &= \mathcal{L}^n\left(\{{\mathbf{M}}_{\alpha}^{2r}\left(\chi_{B_{10r}(x_4)}|\nabla u-\nabla w|^p \right)>\varepsilon^{-a}\lambda\}\cap B_r(x) \right)\\
&\le \frac{C}{(\varepsilon^{-a}\lambda)^{\frac{n}{n-\alpha}}}\left(\int_{B_{10r}(x_4)}{|\nabla u-\nabla w|^p dy} \right)^{\frac{n}{n-\alpha}}\\
&\le \frac{C}{(\varepsilon^{-a}\lambda)^{\frac{n}{n-\alpha}}}r^{\frac{n^2}{n-\alpha}}\left(\fint_{B_{10r}(x_4)}{|\nabla u - \nabla w|^p dy} \right)^{\frac{n}{n-\alpha}}\\
&= \frac{C}{(\varepsilon^{-a}\lambda)^{\frac{n}{n-\alpha}}}r^n \left(r^{\alpha}\fint_{B_{10r}(x_4)}{|\nabla u - \nabla w|^p dy} \right)^{\frac{n}{n-\alpha}}.
\end{split}
\end{align}
As $x_2, x_3$ defined in previous case, since $d(x,\partial\Omega) \le 4r$, we can check easily that
\begin{align*}
B_{10r}(x_4) \subset B_{12r}(x) \subset B_{13r}(x_2),\\
B_{10r}(x_4) \subset B_{12r}(x) \subset B_{13r}(x_3).
\end{align*}
Hence,
\begin{align*}
\displaystyle{r^{\alpha}\fint_{B_{10r}(x_4)}{|\nabla u|^p dy}} \le {{C r^{\alpha}\fint_{B_{13r}(x_2)}{|\nabla u|^p dy} \le C {\bf M M}_{\alpha}(|\nabla u|^p)(x_2)}} \le C\lambda;
\end{align*}
and
\begin{align*}
\displaystyle{r^{\alpha}\fint_{B_{10r}(x_4)}{(|F|^p+|\nabla \sigma|^p)dy}} & \le C r^{\alpha}\fint_{B_{13r}(x_3)}{(|F|^p+|\nabla \sigma|^p)dy} \\ & \le {\mathbf{M}}_{\alpha}(|F|^p+|\nabla \sigma|^p)(x_3) \le C\varepsilon^b\lambda.
\end{align*}
Both of them are applied to \eqref{eq:S1B} to get:
\begin{align*}
\left(r^{\alpha}\fint_{B_{10r}(x_4)} |\nabla u-\nabla w|^p dy \right)^{\frac{n}{n-\alpha}} &\le C \left[\varepsilon^b\lambda+\lambda^{\frac{p-1}{p}}(\varepsilon^b\lambda)^{\frac{1}{p}} \right]^{\frac{n}{n-\alpha}}\\
&\le C\left(\varepsilon^{\frac{b}{p}}\lambda\right)^{\frac{n}{n-\alpha}}\left[ \varepsilon^{\frac{b(p-1)}{p}}+1\right]^{\frac{n}{n-\alpha}}.
\end{align*}
Therefore, 
\begin{align*}
S_1^B \le Cr^n\varepsilon^{\left( a+\frac{b}{p}\right)\frac{n}{n-\alpha}}\left(\varepsilon^{\frac{b(p-1)}{p}}+1 \right)^{\frac{n}{n-\alpha}} \le Cr^n\varepsilon^{\left(a+\frac{b}{p} \right)\frac{n}{n-\alpha}}.
\end{align*}
For the estimation of $S_2^B$, by Lemma \ref{lem:reverseHolderbnd}, the reserve H\"older's inequality is properly utilized to show that there exist $\Theta=\Theta(n,p,\Lambda_1,\Lambda_2,c_0)>p$ and a constant $C = C(n,p,\Lambda_1,\Lambda_2,c_0,\Theta)>0$ giving:
\begin{align}
\label{eq:S2est_bnd}
\begin{split}
S_2^B &\le \frac{Cr^n}{(\varepsilon^{-a} \lambda)^{\frac{\Theta}{p}\frac{n}{n-\alpha}}} \left( \fint_{B_{10r}(x_4)}{|\nabla w|^\Theta dy}\right)^{\frac{n}{n-\alpha}} \\ &\le \frac{Cr^n}{(\varepsilon^{-a} \lambda)^{\frac{\Theta}{p}\frac{n}{n-\alpha}}} \left( \fint_{B_{14r}(x_4)}{|\nabla w|^p dy}\right)^{\frac{\Theta n}{p(n-\alpha)}}.
\end{split}
\end{align}
From the determination of $x_2, x_3$ in previous case, it is a simple matter to get
\begin{align*}
B_{14r}(x_4) \subset B_{16r}(x)\subset B_{17r}(x_2),\\
B_{14r}(x_4) \subset B_{16r}(x)\subset B_{17r}(x_3),
\end{align*}
which implies
\begin{align*}
\displaystyle{r^{\alpha}\fint_{B_{14r}(x_4)}{|\nabla u|^p dy}} \le r^{\alpha}C\fint_{B_{17r}(x_2)}{|\nabla u|^pdy} \le C{\mathbf{M}\mathbf{M}}_\alpha(|\nabla u|^p)(x_2) \le C\lambda,
\end{align*}
and
\begin{align*}
\displaystyle{r^{\alpha}\fint_{B_{14r}(x_4)}{(|F|^p+|\nabla u|^p)dy}} & \le r^{\alpha}\fint_{B_{17r}(x_3)}{(|F|^p+|\nabla u|^p)dy} \\ &\le {\mathbf{M}}_{\alpha}(|F|^p+|\nabla u|^p)(x_3) \le C\varepsilon^b\lambda.
\end{align*}
The use of Lemma \ref{lem:l2} enables us to write:
\begin{align*}
r^{\alpha}\fint_{B_{14r}(x_4)}{|\nabla w|^p dy} &\le C r^{\alpha}\fint_{B_{14r}(x_4)}{|\nabla u|^p dy} + C r^{\alpha}\fint_{B_{14r}(x_4)}{|\nabla u - \nabla w|^p dy}\\ & \le C r^{\alpha}\fint_{B_{20r}(x_4)}{|\nabla u|^p dy} + Cr^{\alpha}\fint_{B_{20r}(x_4)}{(|F|^p+|\nabla \sigma|^p)dy} \\ & \quad + C\left(r^{\alpha}\fint_{B_{20r}(x_4)}{|\nabla u|^p dy} \right)^{\frac{p-1}{p}}\left(r^{\alpha}\fint_{B_{20r}(x_4)}{|\nabla\sigma|^p dy} \right)^{\frac{1}{p}} \\
&\le C\left[\lambda + \varepsilon^b\lambda + \lambda^{\frac{p-1}{p}}(\varepsilon^b\lambda)^{\frac{1}{p}} \right]\\
&= C\lambda\left[1+\varepsilon^b+\varepsilon^{\frac{b}{p}}\right].
\end{align*}
Thus, back to \eqref{eq:S2est_bnd}, we finally get the estimation for $S_2^B$ as follows
\begin{align*}
S_2^B \le Cr^n\varepsilon^{\frac{a\Theta}{p}\frac{n}{n-\alpha}}\lambda^{\frac{-\Theta n}{p(n-\alpha)}}\left[\lambda\left(1+\varepsilon^b+\varepsilon^{\frac{b}{p}}\right) \right]^{\frac{\Theta n}{p(n-\alpha)}} \le Cr^n\varepsilon^{\frac{a\Theta n}{p(n-\alpha)}}.
\end{align*}
Therefore, as in previous case, with the choice of $a = \frac{p(n-\alpha)}{n\Theta}$, $b$ is also taken according to the formula
\begin{align*}
a+\frac{b}{p}=1,
\end{align*}
that giving us evidence of the desired results.
\end{proof}

It remains to prove the final theorem, where we establish the $\mathbf{M}_\alpha$ gradient norm estimate of solution as below. 

As far as Theorem \ref{theo:main2} is applied, the proof of Theorem \ref{theo:regularityMalpha} can be done in the same way as what already obtained in proof of Theorem \ref{theo:regularityM0}.

\begin{proof}[Proof of Theorem \ref{theo:regularityMalpha}]
Let us rephrase the definition of norm in Lorentz space $L^{q,s}(\Omega)$ in \eqref{eq:lorentz} as:
\begin{align*}
\|{\mathbf{M}\mathbf{M}}_\alpha(|\nabla u|^p)\|^s_{L^{s,q}(\Omega)} = q \int_0^\infty{\lambda^s \mathcal{L}^n\left(\{{\mathbf{M}\mathbf{M}}_\alpha(|\nabla u|^p)>\lambda\} \right)^{\frac{s}{q}}\frac{d\lambda}{\lambda}}.
\end{align*}
Changing the variable $\lambda$ to $\varepsilon^{-a}\lambda$ within the integral, we get that:
\begin{align*}
\|{\mathbf{M}\mathbf{M}}_\alpha(|\nabla u|^p)\|^s_{L^{s,q}(\Omega)} = \varepsilon^{-as}q\int_0^\infty{\lambda^s\mathcal{L}^n\left(\{{\mathbf{M}\mathbf{M}}_\alpha(|\nabla u|^p)>\varepsilon^{-a}\lambda\} \right)^{\frac{s}{q}}\frac{d\lambda}{\lambda}}.
\end{align*}
Alternatively, it can be applied Theorem \ref{theo:main2} implies that
\begin{align*}
\mathcal{L}^n\left(\{{\mathbf{M}\mathbf{M}}_\alpha(|\nabla u|^p)>\varepsilon^{-a}\lambda\}\right) &\le C\varepsilon \mathcal{L}^n\left(\{{\mathbf{M}\mathbf{M}_{\alpha}}(|\nabla u|^p)>\lambda\}\cap\Omega \right)\\ &~~+ \mathcal{L}^n\left(\{{\mathbf{M}}_\alpha(|F|^p+|\nabla u|^p)>\varepsilon^b\lambda\}\cap\Omega \right),
\end{align*}
which gives
\begin{align*}
\|{\mathbf{M}\mathbf{M}}_\alpha(|\nabla u|^p)\|^s_{L^{s,q}(\Omega)} &\le C\varepsilon^{-as+\frac{s}{q}}q\int_0^\infty{\lambda^s\mathcal{L}^n\left(\{{\mathbf{M}\mathbf{M}}_\alpha(|\nabla u|^p)>\lambda\}\cap\Omega \right)^{\frac{s}{q}}\frac{d\lambda}{\lambda}}\\
&~~~+ C\varepsilon^{-as}q\int_0^\infty{\lambda^s\mathcal{L}^n\left(\{{\mathbf{M}}_\alpha(|F|^p+|\nabla \sigma|^p)>\varepsilon^b\lambda\}\cap\Omega \right)^{\frac{s}{q}}\frac{d\lambda}{\lambda}}.
\end{align*}
Performing change of variable in the second integral on right-hand side, we get
\begin{align*}
\|{\mathbf{M}\mathbf{M}}_\alpha(|\nabla u|^p)\|^s_{L^{s,q}(\Omega)} &\le C\varepsilon^{-as+\frac{s}{q}}\|{\mathbf{M}\mathbf{M}}_\alpha\left(|\nabla u|^p \right)\|^s_{L^{s,q}(\Omega)}\\ &~~+C\varepsilon^{-as-bs}\|{\mathbf{M}}_\alpha(|F|^p+|\nabla \sigma|^p)\|^s_{L^{s,q}(\Omega)}.
\end{align*}
Therefore, for $0<s<\infty$ and $0<q<\frac{1}{a}=\frac{n}{n-\alpha}.\frac{\Theta}{p}$ it turns out that $s\left(\frac{1}{q}-a \right)>0$.  Then, it is possible to choose $\varepsilon_0>0$ sufficiently small satisfying:
\begin{align*}
C\varepsilon_0^{s\left(\frac{1}{q}-a\right)} \le \frac{1}{2},
\end{align*}
and this finishes the proof, for all $\varepsilon \in (0,\varepsilon_0)$.
\end{proof}


\end{document}